\documentclass[pdflatex,sn-mathphys-num]{sn-jnl}

\usepackage{graphicx}%
\usepackage{multirow}%
\usepackage{amsmath,amssymb,amsfonts}%
\usepackage{amsthm}%
\usepackage{mathrsfs}%
\usepackage[title]{appendix}%
\usepackage{xcolor}%
\usepackage{textcomp}%
\usepackage{manyfoot}%
\usepackage{booktabs}%
\usepackage{algorithm}%
\usepackage{algorithmicx}%
\usepackage{algpseudocode}%
\usepackage{listings}%

\theoremstyle{thmstyleone}%
\newtheorem{theorem}{Theorem}
\newtheorem{proposition}[theorem]{Proposition}%
\newtheorem{lemma}[theorem]{Lemma}

\theoremstyle{thmstyletwo}%
\newtheorem{remark}{Remark}%

\theoremstyle{thmstylethree}%

\usepackage{enumitem}

\renewcommand{\H}{\mathcal{H}}
\newcommand{\tto}{\rightrightarrows}

\makeatletter
\newcommand{\namedlabel}[2]{%
    \begingroup
    #2%
    \def\@currentlabel{#2}%
    \phantomsection
    \label{#1}%
    \endgroup
}
\makeatother

\usepackage{xcolor}
\usepackage[normalem]{ulem}
\usepackage{tikz}
\usetikzlibrary{patterns,decorations.pathmorphing,arrows.meta,calc}

\begin{document}

\title[Volterra Sweeping Process]{Volterra Sweeping Processes with Multivalued Perturbations under Compactness Conditions}


\author[1]{\fnm{Abderrahim} \sur{Jourani}\email{abderrahim.jourani@u-bourgogne.fr}}
\author[2,3]{\fnm{Diana} \sur{Narv\'aez}\email{diana.narvaez@postdoc.uoh.cl}}
\equalcont{These authors contributed equally to this work.}
\author*[2,3]{\fnm{Emilio} \sur{Vilches}\email{emilio.vilches@uoh.cl}}
\equalcont{These authors contributed equally to this work.}
\affil[1]{\orgdiv{Institut de Math\'ematiques de Bourgogne, UMR 5584, CNRS}, \orgname{Universit\'e Bourgogne Europe}, \state{Dijon}, \country{France}}
\affil*[2]{\orgdiv{Instituto de Ciencias de la Ingenier\'ia}, \orgname{Universidad de O'Higgins}, \state{Rancagua}, \country{Chile}}
\affil[3]{\orgdiv{Centro de Modelamiento Matem\'atico}, \orgname{Universidad de Chile}, \state{Santiago}, \country{Chile}}


\abstract{We study integro-differential sweeping processes of Volterra type in a separable Hilbert space, in which the velocity is governed by the normal cone to a prox-regular moving set and is driven by an outer set-valued perturbation together with a history-dependent integral term that endows the dynamics with memory. The perturbation is assumed measurable, with closed convex values, of linear growth, and upper semicontinuous from the strong to the weak topology. We study the existence of absolutely continuous solutions under either of two alternative compactness hypotheses: ball-compactness of the moving sets, or a measure-of-noncompactness condition on the perturbation. After a reduction of the constrained dynamics to an unconstrained differential inclusion and uniform a priori bounds on the state and its velocity, existence follows from a fixed-point theorem for set-valued maps with contractible values. The memory of the process makes this contractibility delicate, and we obtain it through a continuation argument that propagates the history of the dynamics. As applications, we solve a quasistatic frictionless viscoelastic contact problem with long memory and a spatially distributed bioeconomic fishery model with ecological memory, both featuring uncertain set-valued forcing and a moving constraint set that is not ball-compact.}




\maketitle

\section{Introduction}
Given a real Hilbert space $\mathcal{H}$, Moreau’s sweeping process is a differential inclusion involving normal cones to a family of moving closed sets. Since its introduction by J.-J. Moreau \cite{MO1,MO2}, it has become a central framework for the study of constrained dynamical systems arising in contact mechanics, electrical circuits, crowd motion, and related areas (see, e.g., \cite{Acary-Bon-Bro-2011,Maury-Venel}). Moreover, the theory of sweeping processes has been extended to broad classes of convex and nonconvex moving sets, including models involving multivalued perturbations and memory effects. In this paper, we consider the integro-differential sweeping process of Volterra type with an outer multivalued perturbation, namely the dynamical system 
\begin{equation}\label{Integro_Problem} 
\left\{ 
\begin{aligned} \dot{x}(t) & \in -N(C(t); x(t)) + f_1(t,x(t))+\int_{0}^t f_2(t,s,x(s))\, ds + F(t,x(t)) & \textrm{ a.e. } t\in I,\\ 
x(0)&=x_0 \in C(0), 
\end{aligned} 
\right. 
\end{equation}
where $I=[0,T]$, $C(t)\subset \mathcal{H}$ are nonempty closed sets and $N(C(t);x(t))$ denotes the Clarke normal cone to $C(t)$ at $x(t)\in C(t)$. The functions $f_1$ and $f_2$ are measurable in the time variable, locally Lipschitz with respect to the state variable, and satisfy suitable growth conditions. The perturbation term $F$ is a set-valued map from $I \times \mathcal{H}$ into convex  subsets of $\mathcal{H}$.

The relevance of memory-dependent dynamics of this type traces back to Brenier, Gangbo, Savar\'e and Westdickenberg \cite{Brenier2013}, who described sticky particle dynamics through a differential inclusion on the space of transport maps, thereby motivating the study of sweeping processes with memory in applications. From the analytical viewpoint, an initial step toward a rigorous theory was provided by Colombo and Kozaily \cite{Colombo_Kozaily2020}, who established existence and uniqueness results for a class of sweeping processes with Volterra-type perturbations by means of Moreau-Yosida regularization techniques. A further important contribution was made by Bouach, Haddad and Thibault \cite{BouachHaddadThibault2022, Haddad2022}, who developed a comprehensive existence theory for nonconvex integro-differential sweeping processes of Volterra type, establishing an analytical framework in which the interplay between nonconvex geometry and memory effects can be studied. 

More recently, Haddad, Gaouir and Thibault \cite{HaddadGaouirThibault2024, HaddadGaouirThibault2025} obtained existence results for integro-differential sweeping processes with prox-regular sets in Hilbert spaces, including the case of outer multivalued perturbations. We stress, however, that the perturbation in \cite{HaddadGaouirThibault2025} is assumed to be \emph{Lipschitz continuous} (with respect to the Hausdorff distance) with nonconvex values, a regularity requirement that is neither implied by nor implies the standing assumptions of the present paper, where $F$ is only upper semicontinuous (with graph closed in $\mathcal{H}\times\mathcal{H}_w$) and convex-valued, subject to a quantitative measure-of-noncompactness condition. Thus the two frameworks are complementary rather than nested. Recently, Vilches \cite{Vilches2024} proposed an alternative approach to well-posedness based on reparametrization techniques, enhanced versions of Gronwall’s inequality, and fixed-point arguments for history-dependent operators in the case $F\equiv 0$. This framework ensures existence and uniqueness of solutions, together with continuous dependence with respect to the data of the problem. Most closely related to the present work, P\'erez-Aros, Torres-Valdebenito and Vilches \cite{Pedro-Manuel-Emilio-2024} developed a Galerkin-like method for integro-differential inclusions with a possibly unbounded right-hand side and applied it to nonconvex state-dependent Volterra sweeping processes under compactness or monotonicity conditions; in particular, the ball-compact case of our main existence theorem is established there. In this work, we study an existence result for \eqref{Integro_Problem} under either a ball-compactness condition on the moving sets or a compactness condition on the multivalued perturbation. The ball-compact alternative is due to \cite{Pedro-Manuel-Emilio-2024}; the genuinely new contribution is the second alternative, in which $F$ is an upper semicontinuous convex-valued perturbation subject to the quantitative estimate $\gamma(F(t,A))\le k_r(t)\gamma(A)$ on bounded sets, $\gamma$ being a measure of noncompactness. This condition is strictly weaker than requiring $F$ to map bounded sets into relatively compact ones. To the best of our knowledge, no corresponding existence result is available in the literature even under this stronger compactness assumption.  As a first application, we prove the solvability of a quasistatic frictionless contact problem for a viscoelastic body with long memory subjected to uncertain, set-valued external loads. This problem is a genuine instance of the compactness condition on the perturbation, since the associated moving constraint set is not ball-compact; here the estimate holds in the degenerate form $k_r\equiv 0$. A second application, to a spatially distributed fishery model with ecological memory, illustrates the full strength of the condition, since there $k_r$ is a nontrivial positive function and neither the moving set nor the perturbation maps bounded sets into relatively compact ones.

Our approach is based on the analysis of an auxiliary problem obtained by fixing the outer multivalued perturbation. A key step in the proof is a continuous-dependence estimate with respect to the perturbation, combined with a compactness property of the associated solution operator. A priori estimates for both the state and its derivative, obtained via a reduction result, play a central role in the analysis. Existence is then obtained by applying a fixed-point theorem for upper semicontinuous set-valued maps with contractible values to the solution operator of the auxiliary problem, whose values are shown to be contractible through a continuation argument. In the general case, the argument is complemented by a truncation technique. \vspace{0.1cm}

The paper is organized as follows. In Section \ref{Math_Prel}, we introduce the necessary mathematical preliminaries from nonsmooth analysis. In Sections \ref{hipo-sol} and \ref{preliminary_L}, we present the technical assumptions imposed on the data and establish technical results, including fixed point results, estimates involving measures of noncompactness, and a reduction result for Volterra sweeping processes. Section \ref{main_result} establishes the main existence theorem for the integro-differential sweeping process with outer multivalued perturbation. The result is obtained under two alternative compactness assumptions: either the ball-compactness of the moving sets or a compactness condition on the perturbation.  Section~\ref{sec:viscoelastic-contact} illustrates the theory through an application to a quasistatic viscoelastic contact problem with memory, and Section~\ref{sec:fishery-application} through a spatially distributed fishery model with ecological memory. Finally, Section \ref{Concl} concludes the paper with a discussion of the main results and future research perspectives.

\section{Mathematical preliminaries}\label{Math_Prel}
From now on, $\H$ stands for a separable Hilbert space whose norm is denoted by $\Vert \cdot \Vert$. The closed unit ball is denoted by $\mathbb{B}$. The notation $\H_w$ stands for $\H$ equipped with the weak topology, and $x_n \rightharpoonup x$ denotes the weak convergence of a sequence $(x_n)_n$ to $x$.

We set $I=[0,T]$ and denote by $L^1\left(I;\H\right)$ the space of $\H$-valued Lebesgue integrable functions defined over $I$. Moreover, we say that $u\in \operatorname{AC}\left(I;\H\right)$ if there exists $f\in L^1\left(I;\H\right)$ and $u_0\in \H$ such that $u(t)=u_0+\int_{0}^t f(s)ds$ for all $t\in I$. 

Throughout the paper, $(I,\mathcal I)$ is endowed with a complete $\sigma$-finite positive measure, and $\mathcal B(\H)$ denotes the Borel $\sigma$-algebra of $\H$.

 \noindent Given $S\subset \H$ and $x\in S$, we say that  $h\in \H$ belongs to the \emph{Clarke tangent cone} $T(S;x)$ when for every sequence $(x_n)_n$ in $S$ converging to $x$ and every sequence of positive numbers $(t_n)_n$ converging to $0$, there exists  $(h_n)_n$ in $\H$ converging to $h$ such that $x_n+t_nh_n\in S$ for all $n\in \mathbb{N}$. This cone is closed and convex and its negative polar $N(S;x)$ is the \emph{Clarke normal cone} to $S$ at $x\in S$, i.e.,
\begin{equation*}
N\left(S;x\right):=\left\{v\in \H\colon \left\langle v,h\right\rangle \leq 0 \quad  \forall h\in T(S;x)\right\}.
\end{equation*}
As usual, $N(S;x)=\emptyset$ if $x\notin S$. Through that normal cone, the Clarke subdifferential of an extended-real-valued function $f\colon \H\to \mathbb{R}\cup\{+\infty\}$ is defined by
\begin{equation*}
\partial f(x):=\left\{v\in \H\colon (v,-1)\in N\left(\operatorname{epi}f,(x,f(x))\right)\right\},
\end{equation*}
where $\operatorname{epi}f:=\left\{(y,r)\in \H\times \mathbb{R}\colon f(y)\leq r\right\}$ is the epigraph of $f$. When the function $f$ is finite and locally Lipschitz around $x$, the Clarke subdifferential is characterized (see, e.g., \cite{Clarke1998}) in the following simple and amenable way
\begin{equation*}
\partial f(x)=\left\{v\in \H\colon \left\langle v,h\right\rangle \leq f^{\circ}(x;h) \textrm{ for all } h\in \H\right\}, 
\end{equation*}
where
\begin{equation*}
f^{\circ}(x;h):=\limsup_{(t,y)\to (0^+,x)}t^{-1}\left[f(y+th)-f(y)\right],
\end{equation*}
is the \emph{generalized directional derivative} of the locally Lipschitzian function $f$ at $x$ in the direction $h\in \H$.  The function $f^{\circ}(x;\cdot)$ is in fact the support of $\partial f(x)$. That characterization easily yields that the Clarke subdifferential of any locally Lipschitzian function has the important property of upper semicontinuity from $\H$ into $\H_w$.

 \noindent For $S\subset \H$ the distance function to the set $S$ at $x\in \H$ is defined by $d_{S}(x):=\inf_{y\in S}\Vert x-y\Vert$. We denote by $\operatorname{Proj}_{S}(x)$ the possibly empty set $\operatorname{Proj}_{S}(x):=\left\{y\in S\colon d_{S}(x)=\Vert x-y\Vert\right\}$. The equality  $N\left(S;x\right)=\overline{\mathbb{R}_+\partial d_S(x)}^{\ast}$  for  $x\in S$, gives an expression of the Clarke normal cone in terms of the distance function (see, e.g.,  \cite{Clarke1998}). As usual, it will be convenient to write $\partial d(x,S)$ in place of $\partial d\left(\cdot,S\right)(x)$.

 \noindent A vector $v\in \H$
is a proximal normal vector to the subset $S$ of $\H$ at $x\in S$ whenever there exists a constant $\sigma \geq 0$ such that
\begin{equation*}
  \langle v,x'-x \rangle \leq \sigma \|x'-x\|^2 \quad\text{for all }x'\in S.
\end{equation*}
The set of such vectors is the proximal normal cone $N^P(S,x)$ to $S$ at $x$. 
The proximal normal cone enjoys a geometrical characterization (see, e.g., \cite{Clarke1998}) given by the equality
\begin{equation}\label{eq2.5}
N^{P}(S,x)=\{v \in \H:\;\exists \lambda>0\text{ s.t. }x\in \operatorname{Proj}_{S}(x+\lambda v )\}.
\end{equation} 
As in the classical case, the proximal normal cone to the epigraph gives rise to an associated subdifferential. More precisely, for a lower semicontinuous function $f\colon \H \to \mathbb{R}\cup\{+\infty\}$ and $x\in \H$ with $f(x)<\infty$, the proximal subdifferential of $f$ at $x$ is (see, e.g., \cite{Clarke1998}) 
the set 
 $$\partial^{P}f(x) = \{ v\in \H: (v, -1)\in N^{P}(\operatorname{epi} f, (x, f(x)))\},$$
 where $\operatorname{epi} f:= \{(y, r)\in \H\times\mathbb{R}: f(y)\leq r\}$ is the epigraph of $f$.  The uniformity of the positive constant $\lambda$ in (\ref{eq2.5}) for the unit proximal normal
vectors to $S$ leads to the concept of uniformly prox-regular sets. For a given $\rho \in ]0,+\infty ]$
the closed subset $S$ is uniformly $\rho$-prox-regular (see, e.g.,  \cite{P_R_T-2000}) if every unit proximal normal vector to $S$
can be realized by a $\rho$-ball, which can be translated in the fact that for all $x\in S$ and all
$0\neq v \in N^{P}(S,x)$ one has
\begin{equation*}
  \left\langle v ,x'-x \right\rangle
  \leq \frac{\left\Vert v \right\Vert}{2\rho }\left\Vert x'-x\right\Vert ^{2},
\end{equation*}
for all $x'\in S$. We make the convention $\frac{1}{\rho }=0$ for
$\rho =+\infty .$ Recall that for $\rho =+\infty $ the uniform $\rho$-prox-regularity
of the closed set $S$ is equivalent to the convexity of the set.

 The following proposition provides some properties of the Proximal and Clarke subdifferentials
of the distance function $d(\cdot,S)$ when the set $S$ is prox-regular.
It also summarizes some important consequences of the
uniform prox-regularity which will be needed in the paper. For the proofs of these
results we refer to \cite{P_R_T-2000}.

\begin{proposition}
Let $S$ be a nonempty closed subset in $\H$ and let $\rho \in ]0,\infty ]$. Then, the following are
equivalent.
\begin{enumerate}[label=(\alph*)]
\item The set $S$ is uniformly $\rho$-prox-regular.
\item For any point $x'$ in $U_{\rho}(S):=\{x\in \H:\;d(x,S)<\rho\}$ the set
$\mathrm{Proj}_{S}(x')$ is a singleton and the mapping $\operatorname{Proj}_{S}(\cdot)$ is continuous on $U_{\rho}(S)$.
\item The proximal subdifferential of $d(\cdot,S)$ coincides with its Clarke subdifferential on $U_{\rho}(S)$;
\item The distance function $d(\cdot,S)$ is continuously Fr\'echet differentiable
on $U_{\rho}(S)$.
\item The set-valued mapping  $x\rightrightarrows N^P(S, x)\cap\mathbb{B}$ is $1/\rho$-hypomonotone, that is, for all $x_i\in S$ and all $v_i\in N^{P}(S,x_i)\cap \mathbb{B}$, $i=1,2$, one has
\begin{equation*}
\left\langle v_{1}-v_{2},x_{1}-x_{2}\right\rangle \geq -\frac{1}{\rho}\left\Vert
x_{1}-x_{2}\right\Vert ^{2}.
\end{equation*}
\end{enumerate}
\end{proposition}
We now introduce the notion of a measure of noncompactness. We refer to \cite{MR3587794,Deimling1985} for more details.  Let $A$ be a bounded subset of $\H$. We define the \emph{Kuratowski measure of non-compactness of $A$}, $\alpha(A)$, as
\begin{equation*}
    \alpha(A)=\inf\{d>0\colon A \textrm{ admits a finite cover by sets of diameter }\leq d\},
\end{equation*}
and the \emph{Hausdorff measure of non-compactness of} $A$, $\beta(A)$, as
\begin{equation*}
    \beta(A)=\inf\{r>0\colon A \textrm{ can be covered by finitely many balls of radius } r\}.
\end{equation*}
For convenience we define $\alpha(A) = \beta(A) =+\infty$, whenever $A$ is unbounded. 
In Hilbert spaces, the relation between these concepts  is given by the inequality: 
\begin{equation*}
    \sqrt{2}\beta(A)\leq \alpha(A)\leq 2\beta(A) \textrm{ for } A\subset \H \textrm{ bounded.}
\end{equation*}
The following proposition summarizes the main properties of Kuratowski and Hausdorff measures of non-compactness.
\begin{proposition}
    Let $\H$ be an infinite dimensional Hilbert space and $B,B_1$ and $B_2$ be subsets of $\H$. Let $\gamma$ be either the Kuratowski or the Hausdorff measures of non-compactness. Then,
    \begin{enumerate}[label=(\roman{*})]
        \item $\gamma(B)=0$ if and only if $\overline{B}$ is compact;
        \item $\gamma(\lambda B)=|\lambda|\gamma(B)$ for every $\lambda\in \mathbb{R}$;
        \item  $\gamma(B_1+B_2)\leq \gamma(B_1)+\gamma(B_2)$;
        \item $B_1\subset B_2$ implies $\gamma(B_1)\leq \gamma(B_2)$;
        \item $\gamma(\operatorname{conv}B)=\gamma(B)$;
        \item $\gamma(\bar{B})=\gamma(B)$.
    \end{enumerate}
\end{proposition} 
We conclude this section with an estimate for the Hausdorff measure of noncompactness of the image of a bounded set under a Lipschitz mapping. 
\begin{lemma}
\label{lem:proj-mnc}
Let $A\subset\H$ be nonempty and let $f\colon A\to\H$ be $\kappa$-Lipschitz. Then, for every bounded set $B\subset A$,
\[
\beta\bigl(f(B)\bigr)\le \kappa\,\beta(B).
\]
In particular, if $S\subset\H$ is $\rho$-uniformly prox-regular with $\rho\in(0,+\infty)$ and $\lambda\in(0,1)$, then the metric projection $\operatorname{Proj}_S$ is single-valued and $(1-\lambda)^{-1}$-Lipschitz on the tube $U_{\lambda\rho}(S)=\{x\in\H\colon d(x,S)<\lambda\rho\}$, and hence
\[
\beta\bigl(\operatorname{Proj}_S(B)\bigr)\le \frac{1}{1-\lambda}\,\beta(B)
\quad\text{for every bounded set } B\subset U_{\lambda\rho}(S).
\]
\end{lemma}
\begin{proof}
By Kirszbraun's extension theorem, valid between Hilbert spaces (see \cite{Kirszbraun1934}; see also \cite[Chapter~1]{BenyaminiLindenstrauss2000}), $f$ admits a $\kappa$-Lipschitz extension $\widetilde f\colon\H\to\H$ with $\widetilde f=f$ on $A$. Let $r>\beta(B)$ and let $x_1,\dots,x_N\in\H$ be such that $B\subset\bigcup_{i=1}^{N}B(x_i,r)$. Since $\widetilde f$ is defined at every center, $\widetilde f\bigl(B\cap B(x_i,r)\bigr)\subset B\bigl(\widetilde f(x_i),\kappa r\bigr)$ for each $i$, whence $f(B)=\widetilde f(B)\subset\bigcup_{i=1}^{N}B\bigl(\widetilde f(x_i),\kappa r\bigr)$ and therefore $\beta(f(B))\le \kappa r$. Letting $r\downarrow\beta(B)$ proves the first inequality. The second one follows by taking $f=\operatorname{Proj}_S$ on $A=U_{\lambda\rho}(S)$. We stress that the extension step cannot be dispensed with: the centers of an almost-optimal cover of $B$ need not belong to the domain of $f$, and without the extension the covering argument only yields a worse constant (cf. \cite[Lemma~2.7]{Pedro-Manuel-Emilio-2024}).
\end{proof}

\section{Technical assumptions}\label{hipo-sol}
\begin{enumerate}
    \item[\namedlabel{Hf}{$(\mathcal{H}^f_1)$}]  The function $f_1\colon I\times \H\to \H$ satisfies
\begin{enumerate}
        \item For each $x\in \H$, the map $t \mapsto f_1(t,x)$ is measurable.

        \item For all $r>0$, there exists an integrable function $\ell_{r}^{1}\colon I\to \mathbb{R}_+$ such that for all $t\in I$
\begin{align*}
        \Vert f_1(t,x)-f_1(t,y)\Vert \leq \ell_{r}^{1}(t)\Vert x-y\Vert \textrm{ for all } x,y\in r\mathbb{B}.
\end{align*}

        \item There exists a nonnegative integrable function $a_{1}$ such that
\begin{align*}
        \Vert f_1(t,x)\Vert \leq a_{1}(t) (1 + \Vert x\Vert) \textrm{ for all } t\in I \textrm{ and } x\in \H.
\end{align*}
\end{enumerate}
\end{enumerate}

\begin{enumerate}[leftmargin=2.75em]
    \item[\namedlabel{HF}{$(\mathcal{H}^F)$}]   The set-valued map $F\colon I\times \H\rightrightarrows \H$ has nonempty, closed and convex values.
    \end{enumerate}
\begin{enumerate}[leftmargin=4.75em]
        \item[\namedlabel{H1F}{$(\mathcal{H}_1^F)$}] The map $t\rightrightarrows \operatorname{gph}F(t,\cdot)$ is measurable, that is, for every  norm-open set $U\subset \H \times \H$, the set $\{ t\colon \operatorname{gph}F(t,\cdot)\cap U\neq \emptyset \}$ is Lebesgue measurable.
        
        \item[\namedlabel{H2F}{$(\mathcal{H}_2^F)$}] For a.e. $t\in I$, $\operatorname{gph}F(t,\cdot)$ is closed on $\mathcal{H}\times \mathcal{H}_w$.

        \item[\namedlabel{H3F}{$(\mathcal{H}_3^F)$}] There exists nonnegative integrable functions $c$ and $d$ such that
        \begin{equation*}
            \begin{aligned}
                \Vert F(t,x)\Vert:=\sup\{\Vert w\Vert \colon w\in F(t,x)\}\leq c(t) \Vert x\Vert + d(t),\,  x\in \H, \textrm{ a.e. } t\in I.
            \end{aligned}
        \end{equation*}
    \end{enumerate}

\begin{enumerate}[leftmargin=4.15em]
    \item[\namedlabel{H4F}{$(\mathcal{H}_{\mathcal{\textrm{Comp}}}^F)$}]
For all $r>0$, there exists an integrable function $k_r\colon [0,T]\to \mathbb{R}_+$ such that for a.e. $t\in [0,T]$ and $A\subset r\mathbb{B}$, one has
\begin{align*}
    \beta(F(t,A))\leq k_r(t)\beta(A),
\end{align*}
where $\beta$ denotes the Hausdorff measure of noncompactness (the same conclusions hold for the Kuratowski measure). 
\end{enumerate}

\begin{enumerate}[leftmargin=2.6em]
    \item[\namedlabel{Hg}{$(\mathcal{H}^f_2)$}]  The function $f_2\colon I\times I\times \H\to \H$ satisfies
     \end{enumerate}   
\begin{enumerate}
\item[(a)] For each $x\in \H$, the map $(t,s) \mapsto f_2(t,s,x)$ is measurable.
           \item[(b)] For all $r>0$, there exists a nonnegative function $\ell_{r}^{2}\in L^1(I)$  such that for all $(t,s)\in D:=\{(t,s)\in I\times I: s\leq t\}$
\begin{align*}
        \Vert f_2(t,s,x)-f_2(t,s,y)\Vert \leq \ell_{r}^{2}(t)\Vert x-y\Vert \textrm{ for all } x,y\in r\mathbb{B}.
\end{align*}
           \item[(c)] There exists a nonnegative function $a_{2}\in L^1(D)$ such that
      \begin{align*}
        \Vert f_2(t,s,x)\Vert \leq a_{2}(t,s)(1+\Vert x\Vert) \textrm{ for all } (t,s)\in D \textrm{ and } x\in \H.   
      \end{align*}
    \end{enumerate}

\begin{description}[
    style=unboxed,
    font=\normalfont,
    leftmargin=0pt,
    labelindent=0pt,
    labelsep=0.5em
]
    \item[\namedlabel{HC}{$(\mathcal{H}^C)$}]  The set-valued map $C\colon I \rightrightarrows \H$ 
 has nonempty, closed, $\rho$-uniformly prox-regular values and there exists $L_{C}\geq 0$ such that 
\begin{equation*}
\operatorname{Haus}(C(t),C(s)):=\sup_{z\in \H}|d(z,C(t))-d(z,C(s))|\leq L_{C} |t-s| \textrm{ for all } s,t\in I.
\end{equation*}
  \item[\namedlabel{HC3}{$(\mathcal{H}^C_{\textrm{Comp}})$}] The set-valued mapping $C$ satisfies the so-called ball-compact property, i.e., for all $t\in I$, for all $r>0$, the set $r\mathbb{B}\cap C(t)$ is compact. 
\end{description}

\section{Preliminary Lemmas}\label{preliminary_L}
Given a Hilbert space $\H$ and a set-valued map $G\colon [0,T]\times \H \tto \H$, for any measurable function $x\colon [0,T]\to \H$, we define the set of selections of $G$ as
\begin{equation*}
    \operatorname{Sel}_G(x)=\{ f: f \textrm{ is measurable  and } f(t)\in G(t,x(t)) \textrm{ a.e. } t\in [0,T]\}.
\end{equation*}
The following result is a fixed point theorem for set-valued maps with contractible values (see, e.g., \cite[Lemma~1]{MR1669396}).
\begin{lemma}\label{Fixed-point}
Let $X$ be a Banach space, let $\mathcal{K}\subset X$ be a nonempty compact convex set, and let $\mathcal{G}\colon \mathcal{K}\rightrightarrows \mathcal{K}$ be an upper semicontinuous set-valued map with nonempty, closed, and contractible values. Then $\mathcal{G}$ has a fixed point.
\end{lemma}
The next lemma establishes a bound on the Hausdorff measure of non-compactness of a bounded set (see, e.g.,  \cite[p.125]{MR1669396}).
\begin{lemma}\label{Lemma-beta}
For any bounded set $B\subset \H$ and any $\varepsilon>0$ there exists a sequence $(x_k)_k\subset B$ such that
\begin{equation*}
\beta(B)\leq 2\beta(\{x_k\colon k\geq 1\})+\varepsilon.
\end{equation*}
\end{lemma}
The following result provides a reduction principle for Volterra sweeping processes, based on the technique originally introduced by Jourani, Haddad, and Thibault in \cite{Haddad2009}. It associates the constrained dynamics with an unconstrained differential inclusion involving the subdifferential of the distance function, whose trajectories are also solutions to the sweeping process.
We first establish the equivalence between the systems
\begin{equation}\label{Sweeping-Dif}
\left\{
\begin{aligned}
\dot{x}(t)\in & -N(C(t);x(t))+f_1(t,x(t))+\int_0^t f_2(t,s,x(s))\, ds\\
&+F(t,x(t)) & \textrm{ a.e. } t\in I,\\
x(0)=&x_0,
\end{aligned}
\right.
\end{equation}
and \begin{equation}\label{Reduced}
\left\{
\begin{aligned}
\dot{x}(t)\in& -(L_{C} + \eta(t))\partial d_{C(t)}(x(t))+f_1(t,x(t))+\int_{0}^t f_2(t,s,x(s))\, ds \\
&+ F(t,x(t)) & \textrm{ a.e. }  t\in I,\\
x(0)=&x_0,
\end{aligned}
\right.
\end{equation}
where $\eta(t)$ is defined by relation (\ref{eqeta}) below. \\
It should be noted that, in the next result, the convexity of the set-valued map is not needed.
\begin{proposition}\label{Main_Result_Red} 
Assume, in addition to \ref{HC}, \ref{HF}, \ref{Hf} and \ref{Hg}, that $x_{0}\in C(0)$. Let $x\colon I \to \mathcal{H}$ be an absolutely continuous function. Then, $x(\cdot)$ is a solution of the differential inclusion given by \eqref{Sweeping-Dif} if and only if it is a solution of \eqref{Reduced}. Moreover, one has
\begin{equation*}
\begin{aligned}
  \Vert x(t)\Vert &\leq \theta(t) \quad \textrm{ for all } t\in I,\\
   \Vert \dot{x}(t)\Vert &\leq \dot{\theta}(t)=2 \left(\gamma(t) + c(t)\right) \theta(t) + \varepsilon(t)\quad \textrm{ a.e. } t\in I,
  \end{aligned}
\end{equation*}
where 
\begin{equation}\label{eqeta}
\eta(t) := 2\left(c(t) + \gamma(t)\right)\theta(t) + 2 \left(\gamma(t) + d(t)\right),
\end{equation}
and 
\begin{equation*}
\begin{aligned}
\gamma(t)&:=a_{1}(t)+\int_0^t a_{2}(t,s)\, ds,\\ 
\theta(t)&:=\Vert x_0\Vert \exp\left(2\int_0^t (c(s) + \gamma(s))\, ds\right)+\int_0^t \varepsilon(s)\exp\left(2\int_s^t (c(\tau) + \gamma(\tau))\, d\tau\right)\, ds,\\ 
\varepsilon(t) &:= L_{C} + 2 \left(\gamma(t) + d(t)\right).
\end{aligned}
\end{equation*}
\end{proposition}
\begin{proof}
Throughout, for an absolutely continuous curve $x\colon I\to\mathcal H$ we write
\[
r(t):=f_1(t,x(t))+\int_0^t f_2(t,s,x(s))\,ds+g(t),
\]
where $g(t)\in F(t,x(t))$ is the measurable selection associated with the inclusion under
consideration, so that \eqref{Sweeping-Dif} and \eqref{Reduced} read, respectively,
\[
\dot x(t)\in -N(C(t);x(t))+r(t),
\qquad
\dot x(t)\in -(L_C+\eta(t))\,\partial d_{C(t)}(x(t))+r(t),
\quad\text{a.e. }t\in I .
\]
Recall from $(\mathcal H^f_1)$, $(\mathcal H^f_2)$ and $(\mathcal H_3^F)$ that the functions
$a_1(t), c(t), d(t)$ and $\gamma(t)=a_1(t)+\int_0^{t}a_2(t,s)\,ds$ are nonnegative and integrable.

\medskip
\noindent\emph{Step 0: Elementary properties of the majorant.}
By construction $\theta\in\operatorname{AC}(I)$ solves the linear Cauchy problem
\begin{equation}\label{eq:theta-ode}
\dot\theta(t)=2\big(c(t)+\gamma(t)\big)\theta(t)+\varepsilon(t),\qquad \theta(0)=\|x_0\|,
\end{equation}
whose coefficients are nonnegative and integrable; hence $\theta$ is nondecreasing and
$\theta(t)\ge\|x_0\|\ge0$. Since $\varepsilon=L_C+2(\gamma+d)$, equation \eqref{eq:theta-ode}
gives the identity
\begin{equation}\label{eq:key-identity}
L_C+\eta(t)=\dot\theta(t)\qquad\text{for a.e. }t\in I ,
\end{equation}
because $L_C+\eta=L_C+2(\gamma+c)\theta+2(\gamma+d)=2(\gamma+c)\theta+\varepsilon=\dot\theta$.
Finally, we record the growth estimate: for any absolutely continuous map $x$ and a.e.\ $t$,
using $(\mathcal H^f_1)$, $(\mathcal H^f_2)$, $(\mathcal H_3^F)$ and $\|x(s)\|\le\theta(s)+w(s)$
with $w(s):=(\|x(s)\|-\theta(s))_+$ and $\theta$ nondecreasing, we obtain the following majorant of $r(t)$,
\begin{equation}\label{eq:r-bound}
\|r(t)\|\le \tfrac12\eta(t)+\big(a_1(t)+c(t)\big)w(t)+\int_0^t a_2(t,s)\,w(s)\,ds ,
\end{equation}
where we have used  
$$
\tfrac12\eta(t)=(\gamma(t)+c(t))\theta(t)+(\gamma(t)+d(t))
\geq a_1(t)(1+\theta(t))+\int_0^{t}a_2(t,s)(1+\theta(s))\,ds+c(t)\,\theta(t)+d(t).
$$
In particular, if $\|x(s)\|\le\theta(s)$ for all $s\in[0,t]$, then $w\equiv0$ on $[0,t]$ and
\begin{equation}\label{eq:r-half-nu}
\|r(t)\|\le \tfrac12\eta(t).
\end{equation}

\medskip
\noindent\emph{Step 1: A velocity estimate for viable curves.}
Let $x\in\operatorname{AC}(I;\mathcal H)$ satisfy $x(t)\in C(t)$ for all $t$ and
$\dot x(t)+\zeta(t)=r(t)$ with $\zeta(t)\in N(C(t);x(t))$ a.e. $t\in [0,T]$. We claim
\begin{equation}\label{eq:vel-bound}
\|\zeta(t)\|\le L_C+\|r(t)\|,\qquad \|\dot x(t)\|\le L_C+2\|r(t)\|\qquad\text{a.e. }t\in I .
\end{equation}
Fix a Lebesgue point $t$ with $\zeta(t)\neq0$. For $h>0$ small, $x(t-h)\in C(t-h)$ and, by assumption 
$(\mathcal H^C)$, $d(x(t-h),C(t))\le L_C h<\rho$. Let $y_h:=\operatorname{Proj}_{C(t)}(x(t-h))\in C(t)$,
so $\|y_h-x(t-h)\|=d(x(t-h),C(t))\le L_C h$. By $\rho$-prox-regularity, every
$\zeta\in N(C(t);x(t))$ satisfies $\langle\zeta,y-x(t)\rangle\le\frac{\|\zeta\|}{2\rho}\|y-x(t)\|^2$
for all $y\in C(t)$.  Applying to $y=y_h$ and split as
$y_h-x(t)=(y_h-x(t-h))+(x(t-h)-x(t))$, this yields
\[
\langle\zeta(t),x(t-h)-x(t)\rangle
\le \frac{\|\zeta(t)\|}{2\rho}\,\|y_h-x(t)\|^2+\|\zeta(t)\|\,L_C h .
\]
Since $\|y_h-x(t)\|\le L_C h+\int_{t-h}^t\|\dot x\|=O(h)$ and
$h^{-1}\big(x(t-h)-x(t)\big)\to-\dot x(t)$, dividing by $h$ and letting $h\to0^+$ gives
$\langle\zeta(t),\dot x(t)\rangle\ge -L_C\|\zeta(t)\|$. As $\dot x(t)=-\zeta(t)+r(t)$,
\[
-\|\zeta(t)\|^2+\langle\zeta(t),r(t)\rangle\ge-L_C\|\zeta(t)\|,
\]
which implies
$$
\|\zeta(t)\|\le \|r(t)\|+L_C \textrm{ and } \|\dot x(t)\|\le\|\zeta(t)\|+\|r(t)\|\le L_C+2\|r(t)\|.
$$
Finally, If $\zeta(t)=0$ the bounds are
trivial, hence, we have proved \eqref{eq:vel-bound}.

\medskip
\noindent\emph{Step 2: A priori bounds for viable solutions.}
Let $x$ be as in Step 1. Set $\psi(t):=\|x(t)\|$. Then $\dot\psi\le\|\dot x\|\le L_C+2\|r\|$
a.e., and, by the growth conditions,
\[
\dot\psi(t)\le \varepsilon(t)+2\big(a_1(t)+c(t)\big)\psi(t)+2\int_0^t a_2(t,s)\,\psi(s)\,ds
=:\Psi[\psi](t),\qquad \psi(0)=\|x_0\| .
\]
The operator $\Psi$ is order preserving. Since  $\theta$ is nondecreasing and by \eqref{eq:theta-ode} together with
$$\gamma(t)\theta(t)=a_1(t)\theta(t)+\theta(t)\int_0^t a_2(t,s)\,ds
\ge a_1(t)\theta(t)+\int_0^t a_2(t,s)\theta(s)\,ds, $$ one has
$\dot\theta\ge\Psi[\theta]$ with $\theta(0)=\psi(0)$. The comparison principle for Volterra
integro-differential inequalities (see the Gr\"onwall inequality
\cite[Lemma~2.1]{Pedro-Manuel-Emilio-2024}) then yields
$$
\|x(t)\|=\psi(t)\le\theta(t)\qquad\text{for all }t\in I .
$$
Consequently $w\equiv0$, inequalities \eqref{eq:r-half-nu} holds, and \eqref{eq:vel-bound} gives
\begin{equation*}
\begin{aligned}
\|\dot x(t)\| &\le L_C+2\|r(t)\|\\
&\le L_C+2\big[(\gamma(t)+c(t))\theta(t)+(\gamma(t)+d(t))\big]\\
&=2\big(\gamma(t)+c(t)\big)\theta(t)+\varepsilon(t)\quad\text{a.e. }t\in I .
\end{aligned}
\end{equation*}

\medskip
\noindent\emph{Step 3: \eqref{Sweeping-Dif}$\Rightarrow$\eqref{Reduced} and the a priori bounds.}
Let $x$ solve \eqref{Sweeping-Dif}. Since $N(C(t);x(t))=\emptyset$ whenever $x(t)\notin C(t)$,
the inclusion forces $x(t)\in C(t)$ for a.e.\ $t$; as $x$ is continuous and
$\operatorname{gph}C$ is closed (a consequence of $(\mathcal H^C)$), in fact $x(t)\in C(t)$ for
\emph{all} $t\in I$. Thus $x$ is viable and Steps~1-2 give the stated bounds on $\|x(t)\|$ and
$\|\dot x(t)\|$, together with $\|\zeta(t)\|\le L_C+\|r(t)\|\le L_C+\tfrac12\eta(t)\le L_C+\eta(t)$.
Writing $\zeta(t)=\|\zeta(t)\|\,\frac{\zeta(t)}{\|\zeta(t)\|}$ with
$\frac{\zeta(t)}{\|\zeta(t)\|}\in N(C(t);x(t))\cap\mathbb B=\partial d_{C(t)}(x(t))$ (valid for
$x(t)\in C(t)$ by prox-regularity), the bound $\|\zeta(t)\|\le L_C+\eta(t)$ shows
$\zeta(t)\in (L_C+\eta(t))\,\partial d_{C(t)}(x(t))$. Hence $x$ solves \eqref{Reduced}.

\medskip
\noindent\emph{Step 4: \eqref{Reduced}$\Rightarrow$\eqref{Sweeping-Dif}.}
Let $x$ solve \eqref{Reduced}. We first prove
\begin{equation}\label{eq:viability}
x(t)\in C(t)\quad\text{and}\quad \|x(t)\|\le\theta(t)\qquad\text{for all }t\in I .
\end{equation}
Set $\varphi(t):=d(x(t),C(t))$ and $w(t):=(\|x(t)\|-\theta(t))_+$; both are absolutely continuous,
nonnegative and vanish at $t=0$ (as $x_0\in C(0)$ and $\|x_0\|=\theta(0)$). Let
$\Xi:=\varphi+w$ and let $t_\rho:=\inf\{t\in I:\varphi(t)\ge\rho\}$; on $[0,t_\rho)$ we have
$\varphi<\rho$, so, by the prox-regularity property,  $d(\cdot,C(t))$ is continuously differentiable near $x(t)$ whenever
$\varphi(t)>0$, with unit gradient $u(t)=\nabla d_{C(t)}(x(t))$, the unique element of
$\partial d_{C(t)}(x(t))$.

We estimate $\dot\Xi$ a.e.\ on $[0,t_\rho)$, distinguishing the (a.e.\ disjoint) sets on which
$\varphi$ and $w$ vanish or not; on level sets $\{\varphi=0\}$ and $\{w=0\}$ the corresponding
derivative is a.e. zero.\\
\noindent \emph{(i) When $\varphi(t)>0$.} By $(\mathcal H^C)$, $\dot\varphi(t) \le L_C+\langle u(t),\dot x(t)\rangle$;
since $\dot x(t)=-(L_C+\eta(t))u(t)+r(t)$ and $\|u(t)\|=1$,
\begin{align*}
\dot\varphi(t) &\le L_C-(L_C+\eta(t))+\langle u(t),r(t)\rangle\\
& \le -\eta(t)+\|r(t)\|\\
&\le -\tfrac12\eta(t)+(a_1(t){+}c(t))w(t)+\int_0^t a_2(t,s)\,w(s)\,ds,
\end{align*}

where the last inequality is due to \eqref{eq:r-bound}.\\
\noindent \emph{(ii) When $w(t)>0$.} Then $\|x(t)\|>\theta(t)$ and $\dot w=\dot\psi-\dot\theta$. If moreover
$\varphi(t)=0$, Step~1 applies at $t$ and $\|\dot x(t)\|\le L_C+2\|r(t)\|$; if $\varphi(t)>0$, then
$\|\dot x(t)\|\le (L_C+\eta(t))+\|r(t)\|$. Using \eqref{eq:key-identity} in the second case and
\eqref{eq:r-bound} in both,
\[
\dot w(t)\le
\begin{cases}
2(a_1{+}c)w(t)+2\int_0^t a_2\,w\,ds, & \varphi(t)=0,\\
\tfrac12\eta(t)+(a_1{+}c)w(t)+\int_0^t a_2\,w\,ds, & \varphi(t)>0 .
\end{cases}
\]
Adding the relevant contributions on each set, the terms $\pm\tfrac12\eta(t)$ cancel on
$\{\varphi>0,\,w>0\}$, and we obtain, in all cases,
\[
\dot\Xi(t)\le 2\big(a_1(t)+c(t)\big)\,\Xi(t)+2\int_0^t a_2(t,s)\,\Xi(s)\,ds
\qquad\text{a.e. }t\in[0,t_\rho) .
\]
Since $\Xi\ge0$ and $\Xi(0)=0$, the Gr\"onwall inequality
\cite[Lemma~2.1]{Pedro-Manuel-Emilio-2024} (with zero forcing term) forces $\Xi\equiv0$ on
$[0,t_\rho)$; in particular $\varphi\equiv0$ there. By continuity, if $t_\rho\le T$ then
$\varphi(t_\rho)=0<\rho$, contradicting the definition of $t_\rho$. Hence $t_\rho>T$ and
\eqref{eq:viability} holds on all  $I$.

Finally, for $x(t)\in C(t)$ one has
$(L_C+\eta(t))\,\partial d_{C(t)}(x(t))\subset N^P(C(t),x(t))\subset N(C(t);x(t))$, so every
solution of \eqref{Reduced} solves \eqref{Sweeping-Dif}. By Step~3 (or directly by Step~2, which
applies since $x$ is now viable) the a priori bounds on $\|x(t)\|$ and $\|\dot x(t)\|$ hold. This
completes the proof.
\end{proof}
\begin{remark}\label{rem:any-m}
An inspection of the proof shows that the equivalence in Proposition~\ref{Main_Result_Red} remains valid when the coefficient $L_C+\eta(t)$ in \eqref{Reduced} is replaced by any integrable function $m$ with $m(t)\ge L_C+\eta(t)$ for a.e.\ $t\in I$. Indeed, in Step~4 the two occurrences of the coefficient cancel:
\[
\dot\varphi(t)+\dot w(t)\le\bigl[-(m(t)-L_C)+\Vert r(t)\Vert\bigr]+\bigl[(m(t)-L_C-\eta(t))+\Vert r(t)\Vert\bigr]=-\eta(t)+2\Vert r(t)\Vert,
\]
and the remainder of the argument is unchanged. Conversely, for $x\in C(t)$ the set $\partial d_{C(t)}(x)=N^P(C(t),x)\cap\mathbb B$ is star-shaped with respect to the origin, so that $(L_C+\eta(t))\,\partial d_{C(t)}(x)\subset m(t)\,\partial d_{C(t)}(x)$; hence every solution of \eqref{Sweeping-Dif} also solves the reduced problem with the coefficient $m$. The a priori bounds of Proposition~\ref{Main_Result_Red} hold for the solutions of this reduced problem as well.
\end{remark}

\section{Existence results}\label{main_result}

In this section, we will study existence results for the integrodifferential inclusion:
\begin{equation}\label{Problema}
    \left\{
    \begin{aligned}
        \dot{x}(t)  &   \in -N(C(t); x(t))+ f_1(t,x(t))+\int_{0}^t f_2(t,s,x(s))\, ds +F(t,x(t))
        & \textrm{ a.e. } t\in I;\\
        x(0)&=x_0 \in C(0).
    \end{aligned}
    \right.
\end{equation}
Depending on the assumptions of the data, we display two existence results for the above problem: one when the moving sets $C(t)$ are ball-compact, and another when the perturbation term $F$ satisfies an appropriate compactness condition. The first existence result was proved in  \cite{Pedro-Manuel-Emilio-2024}. The second one is new and generalizes several existence results present in the literature.

Recall that the functions $\gamma$, $\varepsilon$ and $\theta$ on $I$ into $\mathbb{R}_+$ were introduced in Proposition~\ref{Main_Result_Red}.

\begin{theorem}\label{existencia-sol} 
Assume that \ref{HC}, \ref{Hf}, \ref{Hg}, and \ref{HF} hold. Assume that one of the following assumption holds:
\begin{itemize}
\item[(a)] The ball-compactness of the moving sets \ref{HC3};
\item[(b)] Compactness condition on the perturbation \ref{H4F}.
\end{itemize}
Then, for any $x_{0}\in C(0)$, there exists a solution $x\in \operatorname{AC}(I;\H)$ of problem \eqref{Problema}. Moreover,
    \begin{equation*}
            \begin{aligned}
            \Vert x(t)\Vert &\leq \theta(t) & \textrm{ for all } t\in I,\\
                \Vert \dot{x}(t)\Vert &\leq 2 \left(\gamma(t) + c(t)\right)\theta(t)+\varepsilon(t) & \textrm{ for a.e. } t\in I.
        \end{aligned}
    \end{equation*}
\end{theorem}
\noindent The proof of this theorem uses the following result, in which the parameters considered are those of Theorem \ref{existencia-sol}.
\begin{proposition}\label{propest}
 Assume \ref{HC}, \ref{Hf}, \ref{Hg} and \ref{H3F}, and let $\rho\in(0,+\infty]$ be the prox-regularity constant in \ref{HC}. Consider integrable functions  $m$, $w_1$, and $w_2$ satisfying 
 $$
 m(t) \geq L_{C} + \eta(t)+d(t)+\bigl(1+\Vert\theta\Vert_\infty\bigr)\gamma(t) \textrm { and } w_1(t), w_2(t) \in d(t)\mathbb{B}, \textrm{ for a.e. } t\in I.
 $$
Set $R_*:=\|\theta\|_\infty$ and write $\ell_1:=\ell^1_{R_*}$ and $\ell_2:=\ell^2_{R_*}$ for the local Lipschitz moduli of $f_1$ and $f_2$ on $R_*\mathbb{B}$ provided by \ref{Hf}(b) and \ref{Hg}(b). Let $x_1(\cdot)$ and $x_2(\cdot)$ be two solutions of
\begin{equation*}
            \begin{aligned}
              \dot{x}(t) &\in  -N_{C(t)}(x(t))+ w_i(t) +f_1(t,x_i(t))+\int_0^t f_2(t,s,x_i(s))\, ds & \textrm{ for a.e. } t\in I,
        \end{aligned}
    \end{equation*}
with the same initial condition. Then for all $t\in [0,T]$, $x_1(t), x_2(t) \in C(t)$ and 
\begin{equation}\label{LipshS}
\Vert x_1(t)-x_2(t)\Vert \leq \int_0^t \exp\Big(\int_s^tv(\tau) d\tau\Big)\Vert w_1(s)-w_2(s)\Vert\, ds
\end{equation}
where
$$v(t) = \frac{m(t)}{\rho}+\ell_1(t)+ \int_0^t\ell_2(t)\, ds.$$
Moreover,  the mapping $S$  defined on $\mathcal{D}_d:=\{w\in L^1([0,T];\H)\colon w(t)\in d(t)\mathbb{B}\ \textrm{a.e.}\}$ with values in $W^{1,1}([0, T ]; \H)$ as $S(w) = x_w(\cdot)$, where $x_w(\cdot)$ is the unique solution of the problem
\begin{align*}
   \dot{x}(t)\in - m(t)\partial d_{C(t)}(x(t)) + w(t) + f_1(t,x(t))+\int_0^t f_2(t,s,x(s))\, ds,
\quad \text{for a.e. } t\in I,
\end{align*}
with $x(0)=x_0\in C(0)$ (for $w\in\mathcal{D}_d$, existence follows from \cite{BouachHaddadThibault2022} together with Remark~\ref{rem:any-m}, and uniqueness from \eqref{LipshS} applied with $w_1=w_2=w$),
satisfies the compactness property
$$
\beta\Big(\{Sw_k(t) : k\geq 1\}\Big) \leq \int_0^t  \exp\Big(\int_s^tv(\tau)\, d\tau\Big) \beta\Big(\{w_k(s) : k\geq 1\}\Big)\, ds\quad \textrm{ for all } t\in [0,T],
$$
where $(w_k)_{k\in \mathbb{N}}$ is  a given sequence in $\mathcal{D}_d$.
\end{proposition}
\begin{proof}   On the one hand, let us prove  \eqref{LipshS}. First, since $x_1$ and $x_2$ solve the sweeping process above and the Clarke normal cone is empty outside $C(t)$, the continuity of $x_i$ and the closedness of $\operatorname{gph}C$ (a consequence of \ref{HC}) give $x_i(t)\in C(t)$ for all $t\in I$. Moreover, Proposition~\ref{Main_Result_Red}, applied with the frozen perturbation $F(t,x):=\{w_i(t)\}$, for which \ref{H3F} holds with $c\equiv0$, yields $\Vert x_i(t)\Vert\le\theta(t)\le R_*$ (the a priori bound is monotone in $c$, so the bound corresponding to $c\equiv0$ is dominated by $\theta$). Hence the Lipschitz moduli $\ell_1,\ell_2$ on $R_*\mathbb{B}$ apply along the trajectories. Finally, by Remark~\ref{rem:any-m}, $x_1$ and $x_2$ also solve the reduced problem with coefficient $m$. Consequently, there are  measurable selections
$v_1$ and $v_2$ of the set-valued mapping $t\mapsto m(t)\partial d_{C(t)}(x_i(t))$, $i=1,2$, respectively, such that
$$
v_i(t) = w_i(t)+f_1(t,x_i(t))+\int_0^t f_2(t,s,x_i(s))\, ds-\dot{x}_i(t)\quad \textrm{ a.e. } t\in [0, T].
$$
\noindent Since $x_1(t),x_2(t)\in C(t)$, we can use the identity
\begin{align*}
\partial d_{C(t)}(x)=N^P(C(t),x)\cap \mathbb{B}
\qquad \text{for } x\in C(t),
\end{align*}
and thus $v_i(t)\in m(t)\left(N^P(C(t),x_i(t))\cap \mathbb{B}\right)$ for $i=1,2$.
From the $\rho$-prox-regularity of $C(t)$, the set-valued mapping
$x \rightrightarrows N^P(C(t),x)\cap \mathbb{B}$ is $1/\rho$-hypomonotone. Hence, applying the hypomonotonicity to the vectors $v_i(t)/m(t)\in N^P(C(t),x_i(t))\cap\mathbb{B}$ and multiplying by $m(t)$,
\begin{align*}
\left\langle v_1(t)-v_2(t), x_1(t)-x_2(t)\right\rangle
\geq -\frac{m(t)}{\rho} \left\|x_1(t)-x_2(t)\right\|^{2}
\end{align*}
or equivalently
\begin{align*}
&\left\langle \dot{x}_{1}(t)-\dot{x}_{2}(t), x_1(t)-x_2(t)\right\rangle
 \leq \frac{m(t)}{\rho} \left\|x_1(t)-x_2(t)\right\|^{2} + \langle w_1(t)-w_2(t), x_1(t)-x_2(t)\rangle + \\
&  \langle f_1(t,x_1(t))-f_1(t, x_2(t))+\int_0^t [f_2(t,s,x_1(s))- f_2(t,s, x_2(s))]\,ds,  x_1(t)-x_2(t)\rangle \\
& \leq \Vert x_1(t)-x_2(t)\Vert \Big( (\frac{m(t)}{\rho}+\ell_1(t)) \left\|x_1(t)-x_2(t)\right\| + \int_0^t \ell_2(t)\Vert x_1(s)-x_2(s)\Vert ds +\\
& \hskip 8cm\Vert w_1(t)-w_2(t)\Vert\Big). 
\end{align*}
Let $\Phi$ be the absolutely continuous function defined on $[0,T]$ by $\Phi(t) = \Vert x_1(t)-x_2(t)\Vert$. Then  
$$\dot{\Phi}(t) \leq  (\frac{m(t)}{\rho}+\ell_1(t)) \Phi(t) + \int_0^t \ell_2(t) \Phi(s)ds + \Vert w_1(t)-w_2(t)\Vert \quad \textrm{ for a.e. } t\in [0,T],$$
which asserts that for all $t\in [0,T]$, one has
\begin{equation*}
\Phi(t) \leq  \int_0^t(\frac{m(s)}{\rho}+\ell_1(s)) \Phi(s) ds + \int_0^t\!\! \int_0^s \ell_2(s)\, \Phi(\tau)\,d\tau\, ds 
+ \int_0^t\Vert w_1(s)-w_2(s)\Vert ds.
\end{equation*}
It remains now to apply the Gr\"onwall lemma (see \cite[Lemma~2.1]{Pedro-Manuel-Emilio-2024}) with 
$$K_1(s) = \Vert w_1(s)-w_2(s)\Vert,\, K_2(s)= \frac{m(s)}{\rho}+\ell_1(s)\, \textrm{ and  }\, K_3(s,\tau) = \ell_2(s);
$$
note that $\Phi$ is continuous and bounded with $\Phi(0)=0$, and that the function $\upsilon$ of that lemma is precisely $\upsilon(t)=K_2(t)+\int_0^t K_3(t,s)\,ds$. This proves \eqref{LipshS}.  On the other hand, we prove the compactness property. Let $(w_k)_k$ be a sequence in
$L^1([0,T];\H)$ and set $x_k:=Sw_k$, so that each $x_k$ solves
\[
\dot x_k(t)\in -m(t)\,\partial d_{C(t)}(x_k(t))+w_k(t)+f_1(t,x_k(t))
+\int_0^t f_2(t,s,x_k(s))\,ds,\qquad x_k(0)=x_0 .
\]
By Proposition~\ref{Main_Result_Red} and Remark~\ref{rem:any-m}, applied with the frozen forcings $w_k\in\mathcal{D}_d$, each $x_k$ is viable, with $\Vert x_k(t)\Vert\le\theta(t)\le R_*$ and velocity dominated by a fixed integrable function; in particular, the family $(x_k)_k$ is bounded in $\operatorname{AC}([0,T];\H)$. Moreover, the
pointwise estimate \eqref{LipshS} holds for every pair $(x_j,x_k)$ with the same exponential
factor $\exp(\int_s^t v)$. The argument now follows the pattern of Lemma~4 in Bothe
\cite{MR1669396}: one discretizes the sweeping process (equivalent, by Remark~\ref{rem:any-m}, to the reduced problem defining $S$) by the catching-up
algorithm, so that, on a partition $0=t_0<\dots<t_N=t$ of mesh $h$, the states are updated
through the metric projections
\[
x_k^{i+1}=\operatorname{Proj}_{C(t_{i+1})}\Big(x_k^{i}+h\,g_k^{i}\Big),
\qquad
g_k^{i}=w_k(t_i)+f_1(t_i,x_k^{i})+\sum_{j<i}h\,f_2(t_i,t_j,x_k^{j})
\]
(as usual for merely integrable data, the nodal values of $w_k$, $f_1$, $f_2$ and of the moduli below stand for the corresponding integral averages over $[t_i,t_{i+1}]$).
At each step, the subadditivity and homogeneity of $\beta$, together with
Lemma~\ref{lem:proj-mnc} (since $x_k^{i}\in C(t_i)$, one has $d\bigl(x_k^{i}+h\,g_k^{i},C(t_{i+1})\bigr)\le L_C h+h\Vert g_k^{i}\Vert$; the discrete a priori bounds and the lower bound imposed on $m$ give $L_C+\sup_k\Vert g_k^{i}\Vert\le m(t_i)+r_i$ with $\sum_i h\,r_i\to0$ as $h\to0^+$, so Lemma~\ref{lem:proj-mnc} applies on the tube $U_{\lambda_i\rho}(C(t_{i+1}))$ with $\lambda_i=(m(t_i)+r_i)h/\rho$ and Lipschitz constant $(1-\lambda_i)^{-1}=1+\tfrac{m(t_i)}{\rho}\,h+o(h)$; here and below, $o(h)$ stands for remainders $h\,\varrho_i^h$ with $\sum_i h\,\varrho_i^h\to0$) and the Lipschitz moduli $\ell_1,\ell_2$
of $f_1,f_2$ (applied through the first part of Lemma~\ref{lem:proj-mnc}, with $A=R_*\mathbb{B}$), yield
\[
\beta\big(\{x_k^{i+1}\}\big)
\le\Big(1+\tfrac{m(t_i)}{\rho}h+h\,\ell_1(t_i)+o(h)\Big)\beta\big(\{x_k^{i}\}\big)
+h\sum_{j<i}h\,\ell_2(t_i)\,\beta\big(\{x_k^{j}\}\big)
+h\,\beta\big(\{w_k(t_i)\}\big).
\]
Here the term $-m(t)\,\partial d_{C(t)}(x_k)$ is never isolated: its effect is entirely
absorbed into the projection, whose measure-of-noncompactness defect is controlled by
Lemma~\ref{lem:proj-mnc}. Letting $h\to0^+$, the convergence of the catching-up scheme for
prox-regular sweeping processes with a Volterra term (see
\cite{BouachHaddadThibault2022,HaddadGaouirThibault2025} and \cite{Pedro-Manuel-Emilio-2024})
turns the discrete recursion into the integral inequality
\begin{equation*}
\begin{aligned}
\beta\big(\{x_k(t):k\ge1\}\big)
\le &\int_0^t\Big(\tfrac{m(\sigma)}{\rho}+\ell_1(\sigma)\Big)\beta\big(\{x_k(\sigma)\}\big)\,d\sigma
+\int_0^t\!\!\int_0^\sigma\ell_2(\sigma)\,\beta\big(\{x_k(s)\}\big)\,ds\,d\sigma\\
&+\int_0^t\beta\big(\{w_k(\sigma)\}\big)\,d\sigma ,
\end{aligned}
\end{equation*}
which is the integral inequality obtained in the first part of the proof for
$\Phi(t)=\|x_1(t)-x_2(t)\|$, now with $\beta(\{x_k(\cdot)\})$ in place of $\Phi$ and
$\beta(\{w_k(\cdot)\})$ in place of $\|w_1(\cdot)-w_2(\cdot)\|$. The function $\Theta(t):=\beta(\{x_k(t):k\ge1\})$ is bounded and continuous, since $|\Theta(t)-\Theta(t')|\le\sup_k\Vert x_k(t)-x_k(t')\Vert$ and the velocities $\dot x_k$ are dominated by a common integrable function, and $\Theta(0)=0$ because all trajectories start at $x_0$. Moreover, $s\mapsto\beta(\{w_k(s)\})$ is measurable (see \cite{Pedro-Manuel-Emilio-2024}). Applying the same Gr\"onwall
inequality \cite[Lemma~2.1]{Pedro-Manuel-Emilio-2024} to $u=\Theta$ (so that $u(0)=0$), with
\[
K_1(s)=\beta\big(\{w_k(s)\}\big),\qquad K_2(s)=\tfrac{m(s)}{\rho}+\ell_1(s),\qquad
 K_3(\sigma,s)=\ell_2(\sigma),
\]
we conclude that, for every $t\in[0,T]$,
\[
\beta\big(\{Sw_k(t):k\ge1\}\big)
\le\int_0^t\exp\Big(\int_s^t v(\tau)\,d\tau\Big)\,\beta\big(\{w_k(s):k\ge1\}\big)\,ds ,
\]
with $v(t)=\tfrac{m(t)}{\rho}+\ell_1(t)+\int_0^t\ell_2(t)\,ds$. This establishes the
compactness property and completes the proof.
\end{proof}
We are now in a position to prove the main result of the paper.
\begin{proof}[Proof of Theorem \ref{existencia-sol}] 
$(a)$: The existence of solutions when the moving sets are ball-compact was proved in \cite[Theorem~5.2]{Pedro-Manuel-Emilio-2024}.\\
$(b)$: The proof is based on ideas developed in \cite{MR1669396}. The proof is divided in two steps:\\
\noindent {\it Step 1:  We assume that the linear growth condition holds with $c\equiv 0$.} Consider the map $S$ defined from $L^1([0,T];\H)$ to $ \operatorname{AC}([0,T];\H)$ as $S(\ell)=x_{\ell}(\cdot)$, where $x_{\ell}(\cdot)$ is the unique solution (see \cite{BouachHaddadThibault2022}; uniqueness also follows from \eqref{LipshS} and Remark~\ref{rem:any-m}) of the problem
\begin{equation*}
\left\{
\begin{aligned}
    \dot{x}(t) &\in -N_{C(t)}(x(t)) + f_1(t, x(t)) + \int_0^t f_2(t, s, x(s))\, ds + \ell(t) \quad \text{ a.e. } t\in [0,T],\\
    x(0)&=x_0\in C(0).
    \end{aligned}
    \right.
\end{equation*}
Put $\alpha(t):= (L_C+2\gamma(t)+2d(t))\exp\Big(2\int_0^t\gamma(s) ds\Big)$.
It is not difficult to see that each solution $R$ of the following differential equation
\begin{equation*}
\dot{y} = 
2\gamma(t)y+\alpha(t)\quad a.e.\ t\in I, \quad y(0)=\Vert x_0\Vert
\end{equation*}
satisfies  $R(t) \geq \theta(t)$, where $\theta $ is defined in Proposition \ref{Main_Result_Red}. From now on, we fix the integrable function
$$m(t):=L_C+\eta(t)+d(t)+\bigl(1+\Vert\theta\Vert_\infty\bigr)\gamma(t),$$
so that Proposition~\ref{propest} applies, and we let $v$ be the corresponding function defined there. Let $$K=\max\Big( \max_{t\in [0,T]}R(t), \exp(\int_0^T v(t)\,dt)\Big).$$ Define the map $S_K$ from $L^1([0,T];\H)$ to $ \operatorname{AC}([0,T];\H)$  by  $S_K(w) = S(\frac{w}{K})$ and the set valued mapping $F_K=KF$. 

To obtain a solution of \eqref{Problema} in this special case, we apply Lemma \ref{Fixed-point} to the set-valued map $\mathcal{G}_K:=S_K\circ \operatorname{Sel}_{F_K}$. 
 Let ${\cal R}_0\subset {\cal C}([0,T];\H)$ be the set defined by
\begin{align*}
{\cal R}_0=\{ f\in {\cal C}([0,T];\H)\colon \Vert f(t)\Vert \leq R(t)\quad \textrm{ for all }  t\in [0,T]\}.
\end{align*}
Let us show that $\mathcal{G}_K({\cal R}_0)\subset {\cal R}_0$. 
Indeed, let $\ell\in {\cal R}_{0}$ and take $w\in \operatorname{Sel}_{F_K}(\ell)$. We will show that $S_K(w) \in {\cal R}_{0}$. As $S_K(w)  = S(\frac{w}{K})$, then $x_{\frac{w}{K}}$ is a solution of the differential inclusion 
\begin{align*}
   \dot{x}(t)\in  -N_{C(t)}(x(t))+ {\frac{w(t)}{K}} +f_1(t,x(t))+\int_0^t f_2(t,s,x(s))\, ds,
\quad \text{for a.e. } t\in I,
\end{align*}
which in turn implies that $x_{\frac{w}{K}}$ is also a solution of the differential inclusion 
\begin{align*}
   \dot{x}(t)\in  -N_{C(t)}(x(t))+ F(t, \ell(t)) +f_1(t,x(t))+\int_0^t f_2(t,s,x(s))\, ds.
\quad \text{for a.e. } t\in I.
\end{align*}
Proposition \ref{Main_Result_Red}  ensure that $\Vert x_{\frac{w}{K}}(t)\Vert \leq  \theta(t)\leq R(t)$. This implies that  $S_K(w) \in {\cal R}_{0}$. \\
\noindent \emph{ Contractibility of the values of $\mathcal{G}_K$:} Fix $\ell\in\mathcal{R}_0$ and set $C_\ell:=\mathcal{G}_K(\ell)
=S_K\bigl(\operatorname{Sel}_{F_K}(\ell)\bigr)$. The measurable selection theorem and \ref{HF} imply that
$\operatorname{Sel}_{F_K}(\ell)\neq\varnothing$. Fix
$\varphi\in\operatorname{Sel}_{F_K}(\ell)$ and put $v_*:=S_K(\varphi)\in C_\ell$. For $\tau\in[0,T]$ and $v\in C_\ell$, consider the continuation
problem
\begin{equation}\label{eq:continuation-contractibility}
\left\{
\begin{aligned}
\dot z(t)\in{}&
-N_{C(t)}(z(t))+f_1(t,z(t))
+\int_0^\tau f_2(t,r,v(r))\,dr
+\int_\tau^t f_2(t,r,z(r))\,dr\\
&+\frac{\varphi(t)}{K} \qquad \qquad \qquad \qquad \qquad \qquad \qquad \qquad \qquad \qquad   \qquad 
 \text{for a.e. }t\in[\tau,T],\\
z(\tau)={}&v(\tau).
\end{aligned}
\right.
\end{equation}
This problem has a unique solution, denoted by $z_{\tau,v}$.
Indeed, choose
$\omega\in\operatorname{Sel}_{F_K}(\ell)$ such that
$v=S_K(\omega)$ and set $\widetilde\omega_\tau
:=\omega\chi_{[0,\tau]}
+\varphi\chi_{(\tau,T]}$.
Then $\widetilde\omega_\tau
\in\operatorname{Sel}_{F_K}(\ell)$, and the restriction of
$S_K(\widetilde\omega_\tau)$ to $[\tau,T]$ solves
\eqref{eq:continuation-contractibility}. Moreover, the continuation is independent of the particular
representation $v=S_K(\omega)$. In fact, if $v=S_K(\omega_1)=S_K(\omega_2)$, then the trajectories generated by
$$
\omega_i\chi_{[0,\tau]}
+\varphi\chi_{(\tau,T]},
\qquad i=1,2,
$$
coincide with $v$ on $[0,\tau]$. On $[\tau,T]$ they solve the same
Volterra sweeping process, with the same past history
$v|_{[0,\tau]}$, the same value $v(\tau)$ at time $\tau$, and the same
forcing $\varphi/K$. Uniqueness of the continuation problem therefore
shows that the two trajectories coincide on $[\tau,T]$.\\
\noindent Define $H\colon[0,T]\times C_\ell\to C_\ell$ by
$$
H(\tau,v)(t):=
\begin{cases}
v(t),&0\leq t\leq\tau,\\
z_{\tau,v}(t),&\tau<t\leq T.
\end{cases}
$$
For  $v=S_K(\omega)$, one has $H(\tau,v)
=
S_K\bigl(
\omega\chi_{[0,\tau]}
+\varphi\chi_{(\tau,T]}
\bigr)$. Hence $H(\tau,v)\in C_\ell$, so $H$ is well defined. Furthermore, $H(0,v)=v_*$ and $H(T,v)=v$ for every $v\in C_\ell$. We now prove that $H$ is continuous. Since $c\equiv0$, every
$\omega\in\operatorname{Sel}_{F_K}(\ell)$ satisfies
\[
\left\|\frac{\omega(t)}{K}\right\|\leq d(t)
\quad\text{for a.e. }t\in[0,T].
\]
Therefore, Proposition~\ref{Main_Result_Red} provides a function
$q\in L^1(0,T)$, independent of $\tau$ and $v$, such that
\[
\left\|
\frac{d}{dt}H(\tau,v)(t)
\right\|
\leq q(t)
\quad\text{for a.e. }t\in[0,T].
\]
For example, one may take $q(t):=
2\gamma(t)R(t)+L_C+2\gamma(t)+2d(t)$. Let $(\tau_n,v_n)\to (\tau,v)$ in $[0,T]\times C_\ell$,  and set $y_n:=H(\tau_n,v_n)$, $y:=H(\tau,v)$. Define
\[
a_n:=\min\{\tau_n,\tau\},
\,
b_n:=\max\{\tau_n,\tau\}, \,\textrm{ and } \, \delta_n:=
\|v_n-v\|_\infty
+2\int_{a_n}^{b_n}q(r)\,dr.
\]
Since $q\in L^1(0,T)$, one has $\delta_n\to0$. From the definition of
$H$ and the common modulus of absolute continuity, it follows that
\begin{equation}\label{eq:before-common-switch}
\|y_n(t)-y(t)\|
\leq\delta_n
\qquad\text{for every }t\in[0,b_n].
\end{equation}
Let $M:=\max_{t\in[0,T]}R(t)$, and let $\ell_M^1,\ell_M^2\in L^1(0,T)$ be the local Lipschitz
moduli of $f_1$ and $f_2$ on $M\mathbb B$. Set
\[
a(t):=\frac{m(t)}{\rho}+\ell_M^1(t),
\]
where $m\in L^1(0,T)$ is the common coefficient supplied by the
reduction argument. For a.e. $t\in[b_n,T]$, both $y_n$ and $y$ are driven by the same
forcing $\varphi/K$. The hypomonotonicity of $x\mapsto N^P(C(t),x)\cap\mathbb B$, 
together with the Lipschitz properties of $f_1$ and $f_2$, gives,  for $D_n(t):=\|y_n(t)-y(t)\|$, the estimate
\begin{equation*}
\dot D_n(t)
\leq
a(t)D_n(t)
+\ell_M^2(t)\int_0^tD_n(r)\,dr
\quad\text{for a.e. }t\in[b_n,T].
\end{equation*}
By \eqref{eq:before-common-switch},
$$
\int_0^tD_n(r)\,dr
\leq
T\delta_n+\int_{b_n}^tD_n(r)\,dr.
$$
Define $Y_n(t):=
D_n(t)+\int_{b_n}^tD_n(r)\,dr$. Then
$$
\dot Y_n(t)
\leq
\bigl(a(t)+1+\ell_M^2(t)\bigr)Y_n(t)
+T\ell_M^2(t)\delta_n \quad \textrm{ for a.e. } t\in [b_n,T],
$$
and $Y_n(b_n)\leq\delta_n$.
Gronwall's inequality yields $\sup_{t\in[b_n,T]}D_n(t)
\leq C_T\delta_n$, where
\[
C_T:=
\bigl(1+T\|\ell_M^2\|_{L^1}\bigr)
\exp\bigl(
\|a+1+\ell_M^2\|_{L^1}
\bigr)
\]
is independent of $n$. Combining with
\eqref{eq:before-common-switch}, we obtain $\|y_n-y\|_\infty\to 0$. Thus $H$ is continuous.\\
\noindent Finally, define $h(s,v):=H(sT,v)$, for  $(s,v)\in[0,1]\times C_\ell$. Then $h$ is continuous and satisfies
\[
h(0,v)=v_*,
\qquad
h(1,v)=v
\quad\text{for every }v\in C_\ell.
\]
Therefore, $C_\ell$ is contractible. Since
$\ell\in\mathcal R_0$ was arbitrary, all the values of
$\mathcal G_K$ are contractible.\vspace{-0.5mm}

\noindent \emph{Construction of a compact set ${\cal R}$:} The aim here is to construct a compact set ${\cal R} \subset {\cal C}([0,T])$ such that $\mathcal{G}_K$ maps ${\cal R}$ into ${\cal R}$ in order to apply the fixed point lemma. So let  $n\in \mathbb{N}$ and  define ${\cal R}_{n+1}:=\overline{\operatorname{co}}\,\mathcal{G}_K({\cal R}_n)$ and set ${\cal R}=\bigcap_n {\cal R}_n$. Then, since ${\cal R}_0$ is convex and $\mathcal{G}_K({\cal R}_0)\subset \mathcal{R}_0$, we get that ${\cal R}_{n+1}\subset {\cal R}_n$ for all $n\in \mathbb{N}$. Moreover, it is clear that ${\cal R}$ is nonempty, convex, closed and uniformly integrable and $\mathcal{G}_K({\cal R})\subset {\cal R}$.\\
\noindent \emph{Claim 1}:  The set ${\cal R}$ is uniformly bounded and equicontinuous.\\ 
\noindent \emph{Proof of Claim 1}: Since ${\cal R}\subset {\cal R}_1=\overline{\operatorname{co}}\,\mathcal{G}_K({\cal R}_0)$, the set ${\cal R}$ is uniformly bounded in ${\cal C}([0,T];\H)$. For equicontinuity, recall from the contractibility argument that every element of $\mathcal{G}_K({\cal R}_0)$ is a trajectory whose velocity is dominated a.e. by the fixed function $q\in L^1(0,T)$, $q(t):=2\gamma(t)R(t)+L_C+2\gamma(t)+2d(t)$. This pointwise bound is preserved under convex combinations and under uniform limits, hence $\Vert \dot u(t)\Vert\le q(t)$ a.e. for every $u\in{\cal R}_1$, and in particular for every $u\in{\cal R}$. Consequently
$$\Vert u(t)-u(t')\Vert\le\int_{t'}^{t} q(\tau)\,d\tau\qquad\text{for all }t,t'\in[0,T]\text{ and all }u\in{\cal R},$$
and since $q\in L^1(0,T)$ its integral is (uniformly) continuous; therefore ${\cal R}$ is equicontinuous. \\
\noindent \emph{Claim 2}:  The set ${\cal R}$ is relatively compact.\\
\noindent \emph{Proof of Claim 2}: To prove that the set ${\cal R}$ is relatively compact, let us consider, for $t\in [0,T]$, the sets ${\cal R}(t):=\{x(t)\colon x\in {\cal R}\}$ and ${\cal R}_n(t):=\{ x(t) : \, x\in {\cal R}_n\}$ for all $n\in \mathbb{N}$. We proceed to show that $\beta({\cal R}(t))=0$ for all $t\in [0,T]$, which, by means of the Arzel\`a-Ascoli theorem (together with Claim~1) would imply the compactness of the set ${\cal R}$. Define $\varrho(t):=\beta({\cal R}(t))$ and $\varrho_n(t):=\beta({\cal R}_n(t))$ (note that $\varrho$, $\varrho_n$ are unrelated to the prox-regularity radius $\rho$). Then, 
\begin{equation*}
\begin{aligned}
\varrho_{n+1}(t)&=\beta(\{u(t)\colon u\in \overline{\operatorname{co}}\,\mathcal{G}_K({\cal R}_n)\})\\
&\leq \beta(\{u(t)\colon u\in \mathcal{G}_K({\cal R}_n)\})\\
&=\beta(\{(S_Kw)(t)\colon w\in \operatorname{Sel}_{F_K}({\cal R}_n)\}).
\end{aligned}
\end{equation*}
Fix $t\in [0,T]$. According to Lemma \ref{Lemma-beta}, for any $\varepsilon>0$ there exists a sequence $\{w^n_k\}_k$ in $\operatorname{Sel}_{F_K}({\cal R}_n)$ such that 
$$
\beta(\{(S_Kw)(t)\colon w\in \operatorname{Sel}_{F_K}({\cal R}_n)\})\leq 2\beta(\{(S_Kw^n_k)(t)\colon k\geq 1\})+\varepsilon.
$$
Since $S_K(w)=S(w/K)$ and $\beta$ is positively homogeneous, applying Proposition~\ref{propest} to the sequence $(w^n_k/K)_k \subset \mathcal{D}_d$ (note that $\Vert w^n_k(t)/K\Vert\le d(t)$ a.e.\ because $c\equiv0$, and that $S$ coincides with the solution operator of Proposition~\ref{propest} for the above choice of $m$, by Remark~\ref{rem:any-m}) yields
\begin{equation*}
\begin{aligned}
\beta(\{(S_Kw^n_k)(t)\colon k\geq 1\})
&=\beta\big(\{S(w^n_k/K)(t)\colon k\geq 1\}\big)\\
&\leq \int_0^t \exp\Big(\int_s^t v(\tau)\,d\tau\Big)\,\beta\Big(\big\{\tfrac{w_k^n(s)}{K} : k\geq 1\big\}\Big)\, ds\\
&= \frac{1}{K}\int_0^t \exp\Big(\int_s^t v(\tau)\,d\tau\Big)\,\beta\big(\{w_k^n(s) : k\geq 1\}\big)\, ds.
\end{aligned}
\end{equation*}
By the very choice $K\geq \exp\big(\int_0^T v(\tau)\,d\tau\big)$ one has $\tfrac{1}{K}\exp\big(\int_s^t v\big)\leq 1$ for all $0\leq s\leq t\leq T$, whence
\begin{equation*}
\begin{aligned}
\beta(\{(S_Kw^n_k)(t)\colon k\geq 1\}) &\leq \int_0^t  
\beta\Big(\{w_k^n(s) : k\geq 1\}\Big)\, ds.
\end{aligned}
\end{equation*}
Moreover, since $w_k^n\in\operatorname{Sel}_{F_K}({\cal R}_n)$, we have
$\{w_k^n(s):k\ge1\}\subset F_K(s,{\cal R}_n(s))$ for a.e. $s\in [0, T]$, hence
\begin{align*}
\beta\big(\{w_k^n(s):k\ge1\}\big)\le \beta\big(F_K(s,{\cal R}_n(s))\big).
\end{align*}
So we have that
\begin{equation*}
\begin{aligned}
\varrho_{n+1}(t)& \leq 2\int_0^t
\beta\Big(F_K(s,{\cal R}_n(s))\Big)\, ds + \varepsilon\\
& =  2K\int_0^t
\beta\Big(F(s,{\cal R}_n(s))\Big)\, ds + \varepsilon\\
&\leq 2K\int_0^t k_K(s)\beta({\cal R}_n(s))\, ds+\varepsilon,\\
\end{aligned}
\end{equation*}
where in the last inequality we have used \ref{H4F} with $r=K$ (recall ${\cal R}_n(s)\subset K\mathbb{B}$). Letting $\varepsilon\downarrow 0$, we obtain, for all $t\in [0,T]$,
$$
\varrho_{n+1}(t)\leq 2K\int_0^t k_K(s)\varrho_n(s)\,ds.
$$
On the other hand, $0\leq \varrho_{n+1}(t)\leq \varrho_n(t)$ for all $n\in \mathbb{N}$. Hence, $\varrho_{\infty}(t):=\lim\limits_{n\to +\infty}  \varrho_n(t)$ is well-defined for all $t\in [0,T]$, $\varrho_{\infty}(0)=0$, and 
for all $t\in [0,T]$
\begin{equation*}
\begin{aligned}
\varrho_{\infty}(t)&\leq 2K\int_0^t k_K(s)\varrho_{\infty}(s)\, ds,
\end{aligned}
\end{equation*}
which (by Gr\"onwall's lemma) implies that $\varrho_{\infty}(t)=0$ for all $t\in [0,T]$. Moreover, since $0\leq \varrho(t)\leq \varrho_n(t)$ for all $n\in \mathbb{N}$, we get that $\varrho(t)=0$ for all $t\in [0,T]$, which implies that ${\cal R}$ is relatively compact.
\vskip 0.1cm
\noindent\emph{Upper semicontinuity of $\mathcal{G}_K$:} Since ${\cal R}$ is compact in ${\cal C}([0,T];\H)$, it suffices to show that $\mathcal{G}_K$ has closed graph. Let $(\ell_n)_n\subset{\cal R}$ with $\ell_n\to\ell$ and let $u_n\in\mathcal{G}_K(\ell_n)$ with $u_n\to u$ in ${\cal C}([0,T];\H)$. Write $u_n=S_K(w_n)$ with $w_n\in\operatorname{Sel}_{F_K}(\ell_n)$, so that $w_n(t)\in F_K(t,\ell_n(t))$ a.e. Because $c\equiv0$ in this step, $\Vert w_n(t)/K\Vert\le d(t)$ a.e. $t\in [0,T]$. Hence $(w_n)_n$ is bounded and uniformly integrable in $L^1([0,T];\H)$, and for a.e.\ $t$ the values $w_n(t)$ lie in the weakly compact set $K d(t)\mathbb{B}$. By the Dunford--Pettis theorem, along a subsequence $w_n\rightharpoonup w$ in $L^1([0,T];\H)$. Using Mazur's lemma to pass to strongly convergent convex combinations, the convexity of the values of $F$, assumption \ref{H2F} (the graph of $F(t,\cdot)$ is closed in $\H\times\H_w$) and $\ell_n(t)\to\ell(t)$, the classical closure theorem for sequences of measurable selections (see, e.g., \cite{Castaing_Valadier_1977, Aubin_Frankowska_2009_book}) yields $w(t)\in F_K(t,\ell(t))$ a.e., i.e., $w\in\operatorname{Sel}_{F_K}(\ell)$. Finally, each $u_n$ solves the reduced inclusion with forcing $w_n/K$ and its velocity is dominated by the common $q\in L^1(0,T)$; passing to the limit as in the proof of \cite[Theorem~5.2]{Pedro-Manuel-Emilio-2024} (upper semicontinuity of the prox-regular normal cone, together with $u_n\to u$ and $w_n/K\rightharpoonup w/K$) shows that $u$ solves the same inclusion with forcing $w/K$. By the uniqueness in Proposition~\ref{propest}, $u=S_K(w)\in\mathcal{G}_K(\ell)$. Thus $\mathcal{G}_K$ has closed graph and, taking values in the compact set ${\cal R}$, is upper semicontinuous with closed values.

So all the hypotheses of Lemma \ref{Fixed-point} are satisfied, and hence there exists a fixed point $x$ of $\mathcal{G}_K$, so that $x$ is a solution of our perturbed-Volterra sweeping process.\\
\noindent\emph{Step 2: The general case, i.e., $c\neq 0$.} Consider the differential inclusion given by
\begin{equation}\label{truncated-problem}
\left\{
\begin{aligned}
    \dot{x}(t) &\in -N_{C(t)}(x(t)) + f_1(t, x(t)) + \int_0^t f_2(t, s, x(s))\, ds + G(t,x(t)) \quad \text{ a.e. } t\in [0,T],\\
    x(0)&=x_0\in C(0).
    \end{aligned}
    \right.
\end{equation}
where $G\colon [0, T]\times \H\rightrightarrows \H$ is a multivalued map defined by
\begin{align*}
G(t,x):=
F\bigl(t,\operatorname{proj}_{\theta(t)\mathbb B}(x)\bigr)
\cap (c(t) \Vert \operatorname{proj}_{\theta(t)\mathbb B}(x)\Vert
+d(t))\mathbb{B}.
\end{align*}
By \ref{H3F}, $\Vert F(t,\operatorname{proj}_{\theta(t)\mathbb B}(x))\Vert\leq c(t)\Vert\operatorname{proj}_{\theta(t)\mathbb B}(x)\Vert+d(t)$, the set $F(t,\operatorname{proj}_{\theta(t)\mathbb B}(x))$ is contained in the ball defining $G$, so the intersection is redundant and $G(t,x)=F(t,\operatorname{proj}_{\theta(t)\mathbb B}(x))$. In particular $G(t,x)$ is nonempty, and, since $\Vert\operatorname{proj}_{\theta(t)\mathbb B}(x)\Vert\leq\min\{\Vert x\Vert,\theta(t)\}$, it satisfies the linear growth condition \ref{H3F} with the same functions $c$ and $d$ as $F$: namely 
$$
\Vert G(t,x)\Vert=\sup\{\Vert w\Vert:w\in G(t,x)\}\leq c(t)\Vert x\Vert+d(t) \textrm{ for all }x\in\H \textrm{ and a.e. } t\in[0,T].
$$

\medskip
\noindent\emph{Step 2.1: Verification of the assumptions of Step~1 for $G$.}
The set-valued mapping $G\colon [0, T]\times \H \rightrightarrows \H$ defined by
\begin{align*}
G(t,x):=
F\bigl(t,\operatorname{proj}_{\theta(t)\mathbb B}(x)\bigr)
\cap (c(t) \Vert \operatorname{proj}_{\theta(t)\mathbb B}(x)\Vert
+d(t))\mathbb{B}
\end{align*}
satisfies
\begin{enumerate}[label=\roman{*})]
    \item $G(t,x)$ has nonempty, convex, and
closed values.

    \item The map $t\rightrightarrows \operatorname{gph} G(t,\cdot)$ is measurable.

    \item For a.e. $t\in I$, $\operatorname{gph} G(t,\cdot)$ is closed on $\H\times \H_w$.

    \item For all $x\in \H$ and a.e. $t\in I$
        \begin{equation*}
            \begin{aligned}
                \Vert G(t,x)\Vert:=\sup\{\Vert w\Vert \colon w\in G(t,x)\}\leq c(t) \theta(t) +d(t).
            \end{aligned}
        \end{equation*}
    \item Compactness assumption on $G\left(\cdot,\cdot\right)$.
\end{enumerate}
By virtue \cite[Lemma 3.2]{Pedro-Manuel-Emilio-2024}, we have $i)$-$iv)$. Moreover, the set-valued mapping $G$ also satisfies $v)$. Indeed, let $r > 0$ and let $A\subset \H$ be a bounded set included in the ball $r\mathbb{B}$. Fix $t\in [0, T]$. Since $\theta(t)\mathbb B$ is a nonempty closed convex subset of $\H$,  the metric projection $\operatorname{proj}_{\theta(t)\mathbb B}:\H\to \theta(t)\mathbb B$ is nonexpansive. Therefore, by \cite[Lemma~2.7]{Pedro-Manuel-Emilio-2024}, we get that
\begin{equation}\label{Proyect}
   \beta\bigl(\operatorname{proj}_{\theta(t)\mathbb B}(A)\bigr) \leq \sqrt{2}\beta(A),
\end{equation}
where $\beta$ is the Hausdorff (or Kuratowski) measure of noncompactness. Since $\theta$ is continuous on $[0,T]$, setting $R_*:=\Vert \theta\Vert_{\infty}$, we have $\operatorname{proj}_{\theta(t)\mathbb B}(A)\subset \theta(t)\mathbb{B}\subset R_{*}\mathbb{B}$.
Therefore, by assumption \ref{H4F} and by definition of $G$ together with \eqref{Proyect}, we have
\begin{align*}
\beta\bigl(G(t,A)\bigr) \leq \beta\bigl(F(t,\operatorname{proj}_{\theta(t)\mathbb{B}}(A))\bigr) \leq k_{R_*}(t)\beta\bigl(\operatorname{proj}_{\theta(t)\mathbb{B}}(A)\bigr) 
\leq \sqrt{2} k_{R_*}(t)\beta(A).
\end{align*}

Thus, $G$ satisfies assumptions $i)$-$v)$. Consequently, the existence of solutions result established in \textit{Step~1} applies to the truncated problem \eqref{truncated-problem}, and therefore there exists an absolutely continuous map $x\colon I\to \H$ solution to \eqref{truncated-problem}.

\medskip
\noindent\emph{Step 2.2: Elimination of the truncation.} Let $x(\cdot)$ be an absolutely continuous solution of the truncated problem \eqref{truncated-problem}. Since $G$ satisfies \ref{H3F} with the same functions $c$ and $d$ as $F$, Proposition~\ref{Main_Result_Red} applied to \eqref{truncated-problem} yields
\begin{align*}
 \Vert x(t)\Vert\leq \theta(t)\quad \text{for all } t\in[0,T].
\end{align*}
Consequently, $\Vert x(t)\Vert\leq\theta(t)$ implies $\operatorname{proj}_{\theta(t)\mathbb B}(x(t))=x(t)$ for all $t\in[0,T]$, and hence, by definition of $G$,
\begin{align*}
G(t,x(t)) &=
F\bigl(t,\operatorname{proj}_{\theta(t)\mathbb B}(x)\bigr)
\cap (c(t) \Vert \operatorname{proj}_{\theta(t)\mathbb B}(x)\Vert
+d(t))\mathbb{B}\\ \\
&=
F(t,x(t))\cap \bigl(c(t)\Vert x(t)\Vert + d(t)\bigr)\mathbb{B}\\ \\
&= F\bigl(t,x(t)\bigr)\quad \text{ for a.e. } t\in [0, T],
\end{align*}
where in the last step we have used \ref{H3F}. Therefore, $x(\cdot)$ is a solution of the original problem
\begin{equation*}
\left\{
\begin{aligned}
    \dot{x}(t) &\in -N_{C(t)}(x(t)) + f_1(t, x(t)) + \int_0^t f_2(t, s, x(s))\, ds + F(t,x(t)) \quad \text{ a.e. } t\in [0,T],\\
    x(0)&=x_0\in C(0).
    \end{aligned}
    \right.
\end{equation*}

\medskip
\noindent\emph{Step 2.3: Estimates for state  and velocity.}
Finally, from Proposition \ref{Main_Result_Red}, we have that
\begin{equation*}
            \begin{aligned}
            \Vert x(t)\Vert &\leq \theta(t) & \textrm{ for all } t\in [0, T].
        \end{aligned}
    \end{equation*}
    Moreover, one has $ \Vert \dot{x}(t)\Vert \leq 2 \left(\gamma(t) + c(t)\right)\theta(t)+\varepsilon(t)$ for a.e. $t\in [0,T]$, which finishes the proof.
\end{proof}

\section{Application to a viscoelastic contact problem with memory}
\label{sec:viscoelastic-contact}

We illustrate Theorem~\ref{existencia-sol} by considering a quasistatic
frictionless contact problem for a viscoelastic body with hereditary memory.
Evolutionary variational inequalities and sweeping-process formulations of
contact problems for viscoelastic materials have been studied extensively; see,
for instance,
\cite{HanSofonea2000,AdlySofonea2019,Vilches2022Contact,
NacrySofonea2022Contact}.

Let $\Omega\subset\mathbb{R}^{d}$, $d\in\{2,3\}$, be a bounded domain with
Lipschitz boundary $\partial\Omega=\overline{\Gamma_D}\cup\overline{\Gamma_N}
\cup\overline{\Gamma_C}$, where the three parts are mutually disjoint and
$\operatorname{meas}(\Gamma_D)>0$. We set
\[
V:=
\left\{
v\in H^{1}(\Omega;\mathbb{R}^{d}):
v=0\ \text{on }\Gamma_D
\right\},
\]
and denote by $\varepsilon(v):=
\frac{1}{2}\bigl(\nabla v+\nabla v^{\top}\bigr)$ the linearized strain tensor. The space $V$ is endowed with its usual norm,
which is equivalent, by Korn's inequality, to $\|v\|_{V}:=
\|\varepsilon(v)\|_{L^{2}(\Omega;\mathbb{S}^{d})}$. Let $A,B\colon V\to V^{*}$ be the operators 
\[
\langle Av,w\rangle
=
\int_{\Omega}
\mathbb{A}(x)\varepsilon(v(x)):\varepsilon(w(x))\,dx \,\textrm{ and }\, \langle Bv,w\rangle
=
\int_{\Omega}
\mathbb{B}(x)\varepsilon(v(x)):\varepsilon(w(x))\,dx.
\]
We assume that $A$ is bounded, symmetric, and strongly
coercive, namely, there exist constants $M_A,m_A>0$ such that
\[
|\langle Av,w\rangle|
\leq M_A\|v\|_{V}\|w\|_{V} \textrm{ and }  \langle Av,v\rangle
\geq m_A\|v\|_{V}^{2}\quad  \textrm{ for all }  v,w\in V.
\]
We also assume that $B$ is bounded. For $(t,s)\in D:=\{(t,s)\in I^{2}:s\leq t\}$, let $C(t,s)\colon V\to V^{*}$ be a bounded linear operator representing the hereditary response of the
material. We assume that $(t,s)\mapsto C(t,s)v$ is measurable for every
$v\in V$ and that there exists $\kappa\in L^{1}(I)$ such that
\begin{equation}\label{eq:memory-operator-bound}
\|A^{-1}C(t,s)\|_{\mathcal{L}(V)}
\leq \kappa(t)
\quad\text{for a.e. }(t,s)\in D.
\end{equation}
The constitutive law corresponding to these operators is formally given by
\[
\sigma(t)
=
\mathbb{A}\varepsilon(\dot u(t))
+\mathbb{B}\varepsilon(u(t))
+\int_{0}^{t}
\mathbb{C}(t,s)\varepsilon(u(s))\,ds.
\]
Thus, the present stress depends both on the current strain and on the past
history of the deformation. Let $\gamma_\nu\colon V\to L^{2}(\Gamma_C)$ denote the normal trace operator.
We assume that the gap between the body and the obstacle is represented by a function $p\in W^{1,\infty}(I;V)$. Define
\[
K_0:=
\left\{
v\in V:
\gamma_\nu v\leq0
\quad\text{a.e. on }\Gamma_C
\right\}
\]
and $K(t):=p(t)+K_0$. Equivalently, $K(t)
=
\left\{
v\in V:
\gamma_\nu v\leq \gamma_\nu p(t)
\quad\text{a.e. on }\Gamma_C
\right\}$. The condition $u(t)\in K(t)$ represents the nonpenetration constraint; the mechanical configuration is illustrated in Figure~\ref{fig:contact-setup}.
\begin{figure}[htbp]
\centering
\begin{tikzpicture}[>={Stealth[length=2.4mm]},line join=round]
  \fill[pattern=north east lines,pattern color=black!45] (-0.7,-0.95) rectangle (6.7,0);
  \draw[very thick] (-0.7,0) -- (6.7,0);
  \node[below] at (3,-0.95) {rigid foundation};
  \fill[pattern=north east lines,pattern color=black!45] (-0.95,0.8) rectangle (0,3.4);
  \draw[very thick] (0,0.8) -- (0,3.4);
  \node[left] at (-0.95,2.1) {$\Gamma_D$};
  \fill[blue!8] (0,0.8) rectangle (6,3.4);
  \draw[blue!55!black,thick] (0,3.4) -- (6,3.4) -- (6,0.8);
  \draw[blue!55!black,thick] (0,0.8) -- (0,3.4);
  \node at (2.5,2.55) {$\Omega$};
  \begin{scope}[thin]
    \draw (4.05,2.75)--(4.6,2.75); \draw (4.05,1.35)--(4.6,1.35);
    \draw (4.325,2.95)--(4.325,2.75); \draw (4.325,1.35)--(4.325,1.15);
    \draw[decorate,decoration={coil,aspect=0.6,segment length=4pt,amplitude=3.2pt}]
        (4.15,2.75) -- (4.15,1.35);
    \draw (4.42,1.35)--(4.42,1.95)--(4.7,1.95)--(4.7,1.35);
    \draw (4.56,2.75)--(4.56,2.10); \draw[line width=1.6pt] (4.44,2.10)--(4.68,2.10);
  \end{scope}
  \draw[blue!70!black,line width=1.4pt] (0,0.8) -- (6,0.8);
  \node[blue!70!black] at (0.75,0.52) {\small$\Gamma_C$};
  \node[right] at (6.02,1.5) {$\Gamma_N$};
  \draw[->,thick] (2.0,0.8) -- (2.0,0.25);
  \node[right] at (2.05,0.52) {$\nu$};
  \draw[{Bar[width=2mm]}-{Bar[width=2mm]}] (6.25,0) -- (6.25,0.8);
  \node[right] at (6.32,0.4) {$g(t)$};
  \foreach \x in {1.1,2.2,3.3,4.4} {\draw[->,thick] (\x,4.2) -- (\x,3.43);}
  \node at (2.75,4.5) {external loads $f(t),\ \mathcal{Q}(t)$};
\end{tikzpicture}
\caption{Quasistatic frictionless contact of a viscoelastic body with memory. The body $\Omega$ is clamped along $\Gamma_D$, subjected to the nominal load $f(t)$ and the uncertain, set-valued load $\mathcal{Q}(t)$ on $\Gamma_N$, and may come into contact with a rigid foundation along $\Gamma_C$, from which it is separated by the gap $g(t)$; here $\nu$ denotes the outward unit normal. The spring-dashpot element depicts the viscoelastic constitutive law with long memory, corresponding to the operators $\mathbb{A}$, $\mathbb{B}$ and the memory kernel $\mathbb{C}(t,s)$.}
\label{fig:contact-setup}
\end{figure}
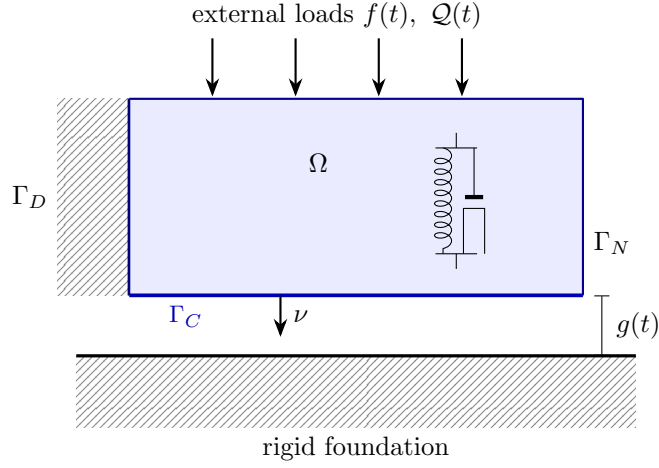
Let $f\in L^{1}(I;V^{*})$ denotes the nominal body and surface forces. To allow
for uncertain or set-valued external loads, let $\mathcal{Q}\colon I\rightrightarrows V^{*}$ be a measurable set-valued mapping with nonempty compact convex values such
that
\begin{equation}\label{eq:load-bound}
\sup_{q\in\mathcal{Q}(t)}
\|A^{-1}q\|_{V}
\leq d(t)
\quad\text{for a.e. }t\in I
\end{equation}
for some $d\in L^{1}(I)$. We consider the following problem: find $u\in\operatorname{AC}(I;V)$ and a measurable selection $q(t)\in\mathcal{Q}(t)$ for a.e. $t\in I$  such that
\begin{equation}\label{eq:contact-membership}
u(t)\in K(t)
\quad\text{for every }t\in I,
\end{equation}
and
\begin{equation}\label{eq:contact-VI}
\begin{aligned}
\Big\langle
&A\dot u(t)+Bu(t)
+\int_{0}^{t}C(t,s)u(s)\,ds
-f(t)-q(t),
v-u(t)
\Big\rangle
\geq0
\end{aligned}
\end{equation}
for every $v\in K(t)$ and for a.e. $t\in I$, with $u(0)=u_0\in K(0)$. Formally, the contact conditions associated with
\eqref{eq:contact-membership}-\eqref{eq:contact-VI} are
\[
u_\nu(t)\leq g(t),\qquad
\sigma_\nu(t)\leq0,\qquad
\sigma_\nu(t)\bigl(u_\nu(t)-g(t)\bigr)=0,
\qquad
\sigma_\tau(t)=0
\quad\text{on }\Gamma_C,
\]
where $g(t):=\gamma_\nu p(t)$. The last relation expresses the absence of
tangential friction.

We now rewrite \eqref{eq:contact-VI} as a Volterra sweeping process. Since
$A$ is symmetric and coercive, $V$ becomes a Hilbert space, denoted by
$\mathcal{H}_A$, when endowed with the scalar product
\[
(v,w)_A:=\langle Av,w\rangle.
\]
Let $N_{K(t)}^{A}$ denote the normal cone to $K(t)$ with respect to this
Hilbertian structure. Then
\begin{equation*}
N_{K(t)}^{A}(u)
=
A^{-1}N_{K(t)}^{V^{*}}(u),
\end{equation*}
where $N_{K(t)}^{V^{*}}(u)
:=
\left\{
\xi\in V^{*}:
\langle\xi,v-u\rangle\leq0
\quad\text{ for all }v\in K(t)
\right\}$. Hence, \eqref{eq:contact-VI} is equivalent to
\begin{equation*}
\begin{aligned}
\dot u(t)\in{}&
-N_{K(t)}^{A}(u(t))
+f_1(t,u(t))
+\int_{0}^{t}f_2(t,s,u(s))\,ds
+F(t,u(t)) \quad \textrm{for a.e. } t\in I, 
\end{aligned}
\end{equation*}
where $f_1(t,v):=A^{-1}f(t)-A^{-1}Bv$, 
 $f_2(t,s,v):=-A^{-1}C(t,s)v$, and $F(t,v):=A^{-1}\mathcal{Q}(t)$.

\begin{proposition}\label{prop:viscoelastic-contact}
Assume that the preceding hypotheses hold. Then, for every
$u_0\in K(0)$, there exist $u\in\operatorname{AC}(I;\mathcal{H}_A)$ and a measurable function $q\colon I\to V^{*}$ satisfying
$q(t)\in\mathcal{Q}(t)$ for a.e. $t\in I$, such that
\eqref{eq:contact-membership} and \eqref{eq:contact-VI} hold.
\end{proposition}

\begin{proof}
For every $t\in I$, the set $K(t)$ is nonempty, closed, and convex.
Therefore, it is uniformly prox-regular with arbitrary prox-regularity
radius. Moreover, since $K(t)=p(t)+K_0$, we have
\[
\operatorname{Haus}_{A}(K(t),K(s))
\leq
\|p(t)-p(s)\|_{A}
\leq
\|\dot p\|_{L^\infty(I;\mathcal{H}_A)}|t-s|.
\]
Thus, the moving set $K$ satisfies $(\mathcal{H}^{C})$.

The function $f_1(t,v)=A^{-1}f(t)-A^{-1}Bv$ is measurable with respect to $t$, globally Lipschitz with respect to $v$,
and satisfies a linear growth condition. Similarly,
\eqref{eq:memory-operator-bound} implies that
\[
\|f_2(t,s,v)-f_2(t,s,w)\|_{A}
\leq
\kappa(t)\|v-w\|_{A} \textrm{ and } \|f_2(t,s,v)\|_{A}
\leq
\kappa(t)(1+\|v\|_{A}).
\]
Hence, $f_1$ and $f_2$ satisfy
$(\mathcal{H}^{f}_{1})$ and $(\mathcal{H}^{f}_{2})$.

The mapping $F(t,v)=A^{-1}\mathcal{Q}(t)$ is independent of $v$ and has nonempty compact convex values. Its graph is
measurable, and \eqref{eq:load-bound} gives
\[
\|F(t,v)\|
\leq d(t) \quad \textrm{ for every }  v\in\mathcal{H}_A \textrm{ and for a.e. } t\in I.
\]
 Therefore,
$(\mathcal{H}^{F})$ holds with $c\equiv0$.

Finally, if $D_0\subset r\mathbb{B}$ is arbitrary, then $F(t,D_0)=A^{-1}\mathcal{Q}(t)$. Since $A^{-1}\mathcal{Q}(t)$ is compact, $\beta(F(t,D_0))=0$. Thus, $(\mathcal{H}^{F}_{\mathrm{Comp}})$ holds in the degenerate form $k_r(t)=0$ for every $r>0$. The conclusion follows from Theorem~\ref{existencia-sol}, applied in the
Hilbert space $\mathcal{H}_A$.
\end{proof}
\begin{remark}
The ball-compactness alternative generally does not apply to this problem.
Indeed, the set $V_0:=
\left\{
v\in V:
\gamma_\nu v=0 \text{ on }\Gamma_C
\right\}$ is an infinite-dimensional closed subspace of $V$, and $p(t)+V_0\subset K(t)$. Consequently, for sufficiently large $r>0$, the set $K(t)\cap r\mathbb{B}$ contains a bounded noncompact subset. Therefore, $K(t)$ is not
ball-compact in general. This shows that Proposition
\ref{prop:viscoelastic-contact} is a genuine application of the compactness
condition imposed on the perturbation $F$, albeit of its degenerate instance $k_r\equiv0$ (a perturbation with values in a fixed compact set). The full strength of $(\mathcal{H}^{F}_{\mathrm{Comp}})$, with a nontrivial modulus $k_r$, is exercised in the fishery model of Section~\ref{sec:fishery-application}.
\end{remark}
\begin{remark}
The deterministic contact problem is recovered by taking $\mathcal{Q}(t)=\{q(t)\}$ for a given $q\in L^{1}(I;V^{*})$. More generally, a finite-dimensional
family of uncertain loads can be considered by setting
\[
\mathcal{Q}(t)
=
\left\{
\sum_{j=1}^{m}\alpha_jq_j(t):
\alpha\in U(t)
\right\},
\]
where $U(t)\subset\mathbb{R}^{m}$ is nonempty compact and convex. In this
case, $\mathcal{Q}(t)$ is compact in $V^{*}$ even though the state space
$V$ is infinite-dimensional.
\end{remark}

\section{A Fishery Model with Ecological Memory}
\label{sec:fishery-application}

Spatial structure and intertemporal exploitation are central features of
bioeconomic fishery models; see, for instance,
\cite{SanchiricoWilen1999}. Modern harvest-control rules are commonly
formulated in terms of biomass limit and target reference points
\cite{KvamsdalEtAl2016}, while viability approaches impose biological and
economic constraints that must be satisfied along the entire evolution
\cite{BeneDoyenGabay2001,CuryEtAl2005}. Memory effects are also relevant in
ecological dynamics because past population and environmental states can
affect present recruitment, resilience, and spatial distribution
\cite{KhalighiEtAl2022,WangFanWang2022}. Motivated by these ingredients, we
consider a spatially distributed fishery whose biomass is constrained to
remain above a time-dependent conservation threshold.

\noindent Let $\Omega\subset\mathbb R^d$ be a bounded measurable set with positive
Lebesgue measure and set $\H:=L^2(\Omega)$. For $x\in\H$, the value $x(\xi)$ represents the biomass density at the spatial
location $\xi\in\Omega$. Let $b\colon I\to\H$ satisfy
\begin{equation*}
 b(t)\geq0\quad\text{a.e. on }\Omega,
 \qquad
 \|b(t)-b(s)\|\leq L_b|t-s|
 \quad\text{for all }s,t\in I,
\end{equation*}
for some $L_b\geq0$. The function $b(t)$ is interpreted as a spatially
varying biomass limit reference profile. Define the moving safe set
\begin{equation*}
 C(t):=\{x\in\H:x(\xi)\geq b(t,\xi)\text{ for a.e. }\xi\in\Omega\}.
\end{equation*}
We next describe the unconstrained biological and management terms. Define
\[
 \psi(z):=\frac{z}{1+|z|},\qquad z\in\mathbb R.
\]
Let $r,\mu\colon I\times\Omega\to\mathbb R_+$ be measurable and assume that
\begin{equation*}
 q(t):=\|r(t,\cdot)\|_{L^\infty(\Omega)}
      +\|\mu(t,\cdot)\|_{L^\infty(\Omega)}
 \in L^1(I).
\end{equation*}
The instantaneous population dynamics are represented by
\begin{equation*}
 f_1(t,x)(\xi)
 :=r(t,\xi)\psi(x(\xi))-\mu(t,\xi)x(\xi).
\end{equation*}
Here, $r$ is the saturating recruitment and $\mu$ the natural
mortality. Let $k\colon D\times\Omega\times\Omega\to\mathbb R$ be measurable and assume
that there exists $\kappa\in L^1(I;\mathbb R_+)$ such that
\begin{equation}\label{eq:fishery-kernel-bound}
 \|k(t,s,\cdot,\cdot)\|_{L^2(\Omega\times\Omega)}\leq\kappa(t)
 \quad\text{for a.e. }(t,s)\in D.
\end{equation}
For $(t,s)\in D$, define the Hilbert-Schmidt operator
$\mathcal K(t,s)\colon\H\to\H$ by
\begin{equation*}
 (\mathcal K(t,s)x)(\xi)
 :=\int_\Omega k(t,s,\xi,\tau)x(\tau)\,d\tau,
\end{equation*}
and set $ f_2(t,s,x):=\mathcal K(t,s)x$. The resulting Volterra term can model delayed recruitment, migration from
previously occupied regions, or accumulated environmental effects.

\noindent To describe uncertain harvesting and finite-dimensional management actions,
let $\underline h,\overline h\colon I\to\mathbb R_+$ be measurable functions
such that
\begin{equation*}
 0\leq\underline h(t)\leq\overline h(t)
 \quad\text{for a.e. }t\in I,
 \qquad
 \overline h\in L^1(I).
\end{equation*}
Let $U\colon I\rightrightarrows\mathbb R^m$ be a measurable set-valued map
with nonempty compact convex values, and let
$B\colon I\to\mathcal L(\mathbb R^m,\H)$ be strongly measurable. Assume that
\begin{equation*}
 d_U(t):=\|B(t)\|\max_{u\in U(t)}|u|\in L^1(I).
\end{equation*}
Define
\begin{equation*}
 F(t,x):=
 \{-h x+B(t)u:
 h\in[\underline h(t),\overline h(t)],\ u\in U(t)\}.
\end{equation*}
The scalar $h$ represents an admissible fishing-mortality rate. The term
$B(t)u$ represents a finite-dimensional family of spatial management
interventions, such as localized restocking, habitat restoration, or
redistribution measures. The managed biomass dynamics are given by
\begin{equation}\label{eq:fishery-sweeping-process}
\left\{
\begin{aligned}
 \dot x(t)\in{}&-N(C(t);x(t))+f_1(t,x(t))
 +\int_0^t f_2(t,s,x(s))\,ds\\
 &+F(t,x(t)) &
 \textrm{ a.e. }t\in I,\\
 x(0)={}&x_0\in C(0).
\end{aligned}
\right.
\end{equation}
The normal-cone term is interpreted as an idealized harvest-control mechanism.
It vanishes while the biomass remains strictly above the conservation
threshold. At the boundary, it removes the outward component of the velocity
that would drive the biomass below the prescribed safe level (see Figure \ref{fig:fishery-sweeping}).
\begin{figure}[t]
\centering
\begin{tikzpicture}[>={Stealth[length=2.2mm]},line join=round,line cap=round]
  \def\btop{4.6}
  \fill[green!12]
    plot[smooth] coordinates {(0,1.5) (1.5,1.62) (3,1.78) (4.5,2.0) (6,2.0) (7,2.0) (8.5,2.2) (10,2.42)}
    -- (10,\btop) -- (0,\btop) -- cycle;
  \draw[->,thick] (-0.15,0) -- (10.7,0) node[right] {$t$};
  \draw[->,thick] (0,-0.15) -- (0,4.85) node[above left] {biomass};
  \node[below] at (10,-0.02) {$T$};
  \draw[dotted] (10,0) -- (10,2.42);
  \draw[green!45!black,thick,densely dashed]
    plot[smooth] coordinates {(0,1.5) (1.5,1.62) (3,1.78) (4.5,2.0) (6,2.0) (7,2.0) (8.5,2.2) (10,2.42)};
  \draw[black!45,dashed]
    plot[smooth] coordinates {(4.5,2.0) (5.2,1.7) (6,1.5) (6.6,1.45) (7,1.55)};
  \draw[blue!55!black,very thick]
    plot[smooth] coordinates {(0,3.2) (1,3.75) (2,4.05) (3,3.55) (4,2.75) (4.5,2.0)};
  \draw[blue!55!black,very thick] (4.5,2.0) -- (7,2.0);
  \draw[blue!55!black,very thick]
    plot[smooth] coordinates {(7,2.0) (8,2.75) (9,3.3) (10,3.65)};
  \foreach \x/\yb in {5.2/1.72, 6/1.52, 6.6/1.47} {\draw[->] (\x,\yb) -- (\x,1.98);}
  \draw[->,orange!85!black,thick] (3.15,3.75) -- (3.15,3.25);
  \draw[->,orange!85!black,thick] (3.85,3.05) -- (3.85,2.55);
  \draw[->,violet!70!black,thick] (7.8,2.2) -- (7.8,2.78);
  \node[green!40!black] at (7.2,4.35) {safe set $C(t)=\{x\ge b(t)\}$};
  \node[blue!55!black,anchor=west] at (2.15,4.2) {biomass $x(t)$};
  \node[green!45!black,anchor=west] at (0.15,1.2) {threshold $b(t)$};
  \node[orange!85!black,anchor=west] at (3.95,3.2) {\small harvesting $-h\,x$};
  \node[align=center] at (6,0.95) {\small harvest control\\[-1pt]$-N(C(t);x)$};
  \node[violet!70!black,anchor=west] at (7.95,2.55) {\small management $B(t)u$};
\end{tikzpicture}
\caption{Sweeping-process dynamics of the fishery model~\eqref{eq:fishery-sweeping-process}. The biomass state $x(t)$, a density field on $\Omega$, is constrained to the safe set $C(t)=\{x\ge b(t)\}$ lying above the time-dependent conservation threshold $b(t)$. Saturating recruitment drives the initial growth and harvesting $-h\,x$ drives the decline. Once the biomass reaches the threshold, the harvest-control mechanism $-N(C(t);x)$ prevents it from dropping below the safe level, whereas the dashed curve indicates the unconstrained tendency that would violate the constraint; finite-dimensional management $B(t)u$ then supports the recovery. The hereditary term $\int_0^t f_2(t,s,x(s))\,ds$ encodes the ecological memory of the system.}
\label{fig:fishery-sweeping}
\end{figure}
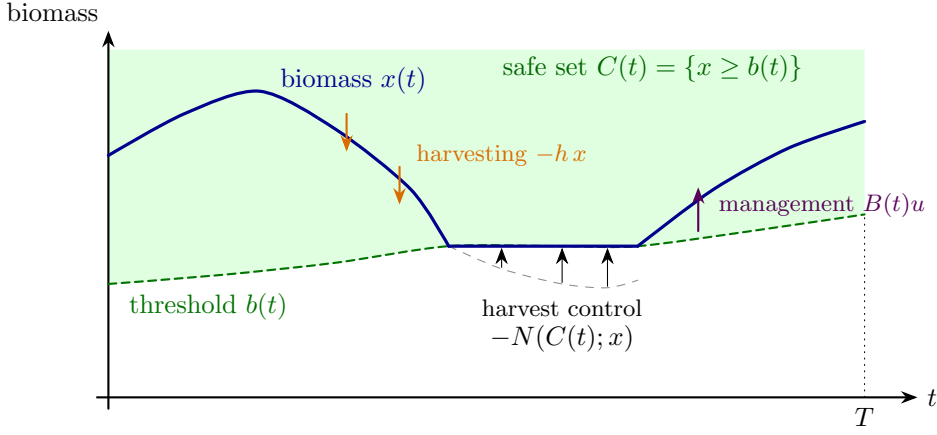
\begin{proposition}\label{prop:fishery-existence}
Under the preceding assumptions, for every $x_0\in C(0)$,
problem~\eqref{eq:fishery-sweeping-process} admits an absolutely continuous
solution.
\end{proposition}

\begin{proof}
For every $t\in I$, the set $C(t)$ is nonempty, closed, and convex. Hence it
is uniformly prox-regular with arbitrary prox-regularity radius. Let
$x\in C(t)$. The projection of $x$ onto $C(s)$ is given pointwise by $ P_{C(s)}x=\max\{x,b(s)\}$. Consequently,
\[
 d(x,C(s))
 =\|(b(s)-x)_+\|
 \leq\|(b(s)-b(t))_+\|
 \leq\|b(s)-b(t)\|.
\]
Interchanging $s$ and $t$ yields $ \operatorname{Haus}(C(t),C(s))
 \leq\|b(t)-b(s)\|
 \leq L_b|t-s|$. Thus, $(\mathcal H^C)$ holds. The scalar function $\psi$ is bounded by one and is globally Lipschitz with
Lipschitz constant one. Therefore,
\[
 \|f_1(t,x)-f_1(t,y)\|
 \leq q(t)\|x-y\| \,\textrm{ and }\,  \|f_1(t,x)\|
 \leq |\Omega|^{1/2}\|r(t,\cdot)\|_{L^\infty(\Omega)}
 +\|\mu(t,\cdot)\|_{L^\infty(\Omega)}\|x\|.
\]
Hence, $(\mathcal H_1^f)$ is satisfied. By the Hilbert-Schmidt estimate and
\eqref{eq:fishery-kernel-bound},
\[
 \|f_2(t,s,x)-f_2(t,s,y)\|
 \leq\kappa(t)\|x-y\| \,\textrm{ and }\,  \|f_2(t,s,x)\|
 \leq\kappa(t)\|x\|
 \leq\kappa(t)(1+\|x\|).
\]
Since $\kappa\in L^1(I)$,  hypothesis
$(\mathcal H_2^f)$ holds.\\
\noindent For each $(t,x)$, the set $F(t,x)$ is nonempty, compact, and convex. Standard
measurability results for images of measurable multifunctions under
Carath\'eodory maps imply $(\mathcal H_1^F)$; see, for example,
\cite{Aubin_Frankowska_2009_book}. To verify $(\mathcal H_2^F)$, fix $t$
outside a null set and suppose that $x_n\to x$, $y_n\rightharpoonup y$, and  $y_n\in F(t,x_n)$. Write $ y_n=-h_nx_n+B(t)u_n$, $h_n\in[\underline h(t),\overline h(t)]$, and $u_n\in U(t)$. The compactness of
$[\underline h(t),\overline h(t)]\times U(t)$ provides a subsequence such
that $h_n\to h$ and $u_n\to u$. Hence, $ y_n\to-hx+B(t)u$ strongly in $\H$, and therefore $y=-hx+B(t)u\in F(t,x)$.
Moreover,
\[
 \|F(t,x)\|
 \leq\overline h(t)\|x\|+d_U(t),
\]
so $(\mathcal H_3^F)$ holds with
$c=\overline h$ and $d=d_U$.\\
\noindent   It remains to verify $(\mathcal H_{\mathcal{\textrm{Comp}}}^F)$. Fix
$r>0$ and a set $A\subset r\mathbb B$. Set $ K_t:=B(t)U(t)$. The set $K_t$ is compact. Since the interval
$[\underline h(t),\overline h(t)]$ is compact, the standard properties of
the Hausdorff (or Kuratowski) measure of noncompactness $\beta$ give
\[
 \beta\bigl(
 \{-h x:h\in[\underline h(t),\overline h(t)],\ x\in A\}
 \bigr)
 \leq\overline h(t)\beta(A).
\]
Indeed, one first covers the compact interval by finitely many intervals of
arbitrarily small diameter and then uses the homogeneity of $\beta$ on each
fixed scalar multiple of $A$. Since addition of a compact set does not
increase the measure of noncompactness, it follows that
$$
 \beta(F(t,A))
 \leq
 \beta\bigl(
 \{-h x:h\in[\underline h(t),\overline h(t)],\ x\in A\}
 \bigr)+\beta(K_t)
 \leq\overline h(t)\beta(A).
$$
Thus, $(\mathcal H_{\mathcal{\textrm{Comp}}}^F)$ holds with $ k_r(t)=\overline h(t)$.
The conclusion now follows from Theorem~\ref{existencia-sol}.
\end{proof}

\begin{remark}
This application genuinely uses alternative~$(b)$ of
Theorem~\ref{existencia-sol}. Indeed, $C(t)$ is not ball-compact. To see this,
fix a measurable set $E\subset\Omega$ with $|E|>0$ and let $(r_n)_n$ be a
Rademacher sequence on $E$. For any $\delta>0$, set $ x_n:=b(t)+\delta(1+r_n)\chi_E$. Then $x_n\in C(t)$ for every $n$, the sequence $(x_n)_n$ is bounded in
$L^2(\Omega)$, and
\[
 \|x_n-x_m\|=\delta\sqrt{2|E|}
 \quad\text{whenever }n\neq m.
\]
Therefore, $C(t)\cap R\mathbb B$ is noncompact for all sufficiently large
$R$. Furthermore, if $\overline h(t)>0$, fix $u_0\in U(t)$. Then
$F(t,\mathbb B)$ contains the translated ball
$-\overline h(t)\mathbb B+B(t)u_0$, and hence the multifunction
$F(t,\cdot)$ does not map bounded sets into relatively compact sets. Thus, the result is
not a consequence of compactness of the moving set or compactness of the
perturbation. What is used is the quantitative estimate $\beta(F(t,A))\leq\overline h(t)\beta(A)$.
\end{remark}

\begin{remark}
For a fish price $p\colon I\to\mathbb R_+$ and a management cost
$\mathfrak c$, one may associate with a feasible triple $(x,h,u)$ the
discounted return
\[
 J(x,h,u)
 :=\int_0^T e^{-\delta t}
 \left(
 p(t)h(t)\int_\Omega x(t,\xi)\,d\xi
 -\mathfrak c(t,h(t),u(t))
 \right)dt.
\]
Proposition~\ref{prop:fishery-existence} guarantees nonemptiness of the
feasible set for the corresponding bioeconomic management problem. The
existence of an optimal management policy would require additional compactness
and lower-semicontinuity assumptions and is not addressed here.
\end{remark}

\section{Conclusions and Future Work}\label{Concl}

In this paper we established the existence of absolutely continuous solutions for
integro-differential sweeping processes of Volterra type driven by an outer set-valued
perturbation, in a separable Hilbert space and with prox-regular moving sets. The
perturbation was only required to be measurable, of linear growth, with nonempty closed
convex values, and upper semicontinuous in the weak sense that its graph is closed when
the state carries the norm topology and the values carry the weak topology. Our analysis
rests on three ingredients: a reduction of the constrained dynamics to an unconstrained
differential inclusion, together with uniform a priori bounds on the state and its
velocity; a continuous-dependence estimate for the associated solution operator,
complemented by a contraction property expressed through a measure of noncompactness; and
a fixed-point theorem for upper semicontinuous set-valued maps with contractible values,
whose applicability was secured by a continuation argument establishing the
contractibility of the solution set of the auxiliary problem. Existence was obtained under
either of two alternative compactness hypotheses: ball-compactness of the moving sets, or
a quantitative measure-of-noncompactness condition on the perturbation. In the latter
case, the general result was reached by means of a truncation technique.

The scope and the limitations of these results delineate several concrete avenues for
future research, of which we regard the following as the most natural.

A first question is whether the two compactness hypotheses considered here can be unified
and relaxed. These alternatives, and the Lipschitz framework of Haddad, Gaouir and
Thibault \cite{HaddadGaouirThibault2025}, are complementary rather than nested. It would
therefore be desirable to formulate a single structural condition on the perturbation, of
dissipative or one-sided Lipschitz type and possibly combined with a partial compactness
requirement, that subsumes the upper semicontinuous, convex-valued setting of the present
work as well as the Lipschitz, nonconvex-valued setting of
\cite{HaddadGaouirThibault2025}, thereby clarifying the exact interplay between regularity
and compactness needed for well-posedness.

A second direction leads beyond prox-regular constraints. The extension of the present
framework to Volterra-perturbed evolution inclusions governed by time-dependent maximal
monotone operators appears promising, and the recent work of Azzam and collaborators
\cite{azzam2025volterra} provides a natural starting point. The central difficulty is to
identify structural properties of the operators under which the reduction and the
contractibility arguments developed herein survive, since prox-regularity of the sets is
used at several points. Closely related is the treatment of state-dependent moving sets
$C(t,x)$ in the presence of the memory term, where even the reduction step must be
reconsidered.

The formulation adopted here also opens control-theoretic questions. The set-valued
perturbation admits a natural interpretation as an admissible control acting on the drift,
which suggests the study of associated optimal control problems and the derivation of
necessary optimality conditions. The discretization and control theory developed for
truncated integro-differential sweeping processes in
\cite{BouachHaddadThibault2022,Haddad2022} indicates a viable route, and the bioeconomic
fishery model of Section~\ref{sec:fishery-application}, in which the harvesting policy is
the decision variable, provides a concrete testbed.

Since the applications considered here are governed by uncertain data, a further natural
step is to develop a genuinely stochastic counterpart in which the perturbation is replaced
or augmented by a random forcing; establishing well-posedness under
measure-of-noncompactness conditions in that setting appears to be open. In a different
direction, the Volterra structure invites the replacement of the memory kernel by a
nonlocal, fractional-in-time operator of Caputo type, together with the identification of
the kernels for which the a priori estimates and the compactness arguments of Section
\ref{preliminary_L} remain valid.

Finally, from a computational standpoint, the design and convergence analysis of
catching-up time-stepping schemes that simultaneously resolve the moving constraint, the
history-dependent integral term, and the set-valued perturbation would complement the
existence theory developed here, and would render the contact and fishery models of
Sections~\ref{sec:viscoelastic-contact} and \ref{sec:fishery-application} amenable to
simulation. Quantitative error estimates in the spirit of the Galerkin-like analysis of
\cite{Pedro-Manuel-Emilio-2024} would be particularly valuable in this respect.

We emphasize that the results obtained here concern existence rather than uniqueness: for a
merely upper semicontinuous, set-valued perturbation, uniqueness is neither expected nor, in
general, true, in contrast with the single-valued case $F\equiv 0$ treated in
\cite{Vilches2024}. Characterizing the perturbations for which the solution set enjoys
additional structural properties, such as compactness, connectedness, or continuous
dependence with respect to the data, constitutes a further question of independent interest.

\backmatter

\section*{Declarations}

\subsection*{Funding} Diana Narv\'aez was supported by ANID (Chile) through Fondecyt Postdoctoral Grant No.~3240146. Emilio Vilches was supported by ANID (Chile) through Fondecyt Regular Grants Nos.~1240120 and 1261728, the CMM BASAL Fund for Centers of Excellence (FB210005), and the ECOS-ANID project ECOS230027.

\subsection*{Conflict of interest} Not applicable.

\bibliography{sn-bibliography}

\end{document}